\documentclass[12pt,a4paper]{article}
\usepackage{amsthm,upref}
\usepackage[theoremfont]{newtxtext}
\usepackage{newtxmath}
\usepackage{microtype}
\usepackage{enumerate}
\usepackage{authblk}
\usepackage[margin=2.8cm]{geometry}
\usepackage{hyperref}
\usepackage{tikz}
\usetikzlibrary{patterns}

\title{Logistic growth in seasonally changing environments}
\author{Daniel Daners}%
\author{Zeaiter Zeaiter}%
\affil{School of Mathematics and Statistics, University of Sydney,
  NSW 2006, Australia\authorcr%
  \nolinkurl{daniel.daners@sydney.edu.au}, \nolinkurl{zeaiter.zeaiter@sydney.edu.au}}%

\date{\today}

\makeatletter
\hypersetup{pdfauthor={Daniel Daners, Zeaiter Zeaiter}}
\hypersetup{pdftitle={\@title}}
\makeatother

\numberwithin{equation}{section}
\numberwithin{figure}{section}

\theoremstyle{plain}
\newtheorem{theorem}{Theorem}[section]
\newtheorem{lemma}[theorem]{Lemma}
\newtheorem{proposition}[theorem]{Proposition}
\newtheorem{corollary}[theorem]{Corollary}

\theoremstyle{definition}
\newtheorem{definition}[theorem]{Definition}
\newtheorem{assumption}[theorem]{Assumption}

\theoremstyle{remark}
\newtheorem{remark}[theorem]{Remark}

\DeclareMathOperator{\dist}{dist}
\DeclareMathOperator{\interior}{int}
\DeclareMathOperator{\supp}{supp}
\DeclareMathOperator{\spr}{r}

\DeclareMathOperator{\repart}{Re}
\DeclareMathOperator{\aaa}{\mathfrak{a}}

\let\oldthebibliography\thebibliography
\renewcommand\thebibliography[1]{%
  \oldthebibliography{#1}
  \setlength{\parskip}{.2ex}
  \setlength{\itemsep}{0pt plus 0.3ex}
  \small
}

\begin{document}
\maketitle
\renewcommand{\thefootnote}{}%

\begin{abstract}
  We consider a parameter dependent periodic-logistic problem with a logistic term involving a degeneracy that replicates time dependent refuges in the habitat of a population. Working under no or very minimal assumptions on the boundary regularity of the domain we show the existence of a time-periodic solution which bifurcates with respect to the parameter and show their stability. We show that under suitable assumptions that the periodic solution blows up on part of the domain and remains finite on other parts when the parameter approaches a critical value.
\end{abstract}

\section{Introduction}
\label{sec:introduction}
Following the poineering work by Hess \cite{hess:91:ppb} and in particular Du and Peng \cite{du:12:ple,du:13:ssp} we consider the existence of non-trivial positive solutions of the periodic-parabolic logistic equation
\begin{equation}
  \label{eq:pp-logistic}
  \begin{aligned}
    \frac{\partial u}{\partial t}+\mathcal A(t)u & =\mu u-b(x,t)g(x,t,u)u &  & \text{in $\Omega\times(0,T)$,} \\ \mathcal B(t)u & =0 &  & \text{on $\partial\Omega\times[0,T]$,} \\ u(\cdot\,,0) & =u(\cdot\,,T) &  & \text{in $\Omega$.}
  \end{aligned}
\end{equation}
The bounded domain $\Omega\subseteq\mathbb R^N$ represents the habitat of some species with population density $u(x,t)$ at location $x\in\Omega$ and time $t\in\mathbb R$. The time $T$ is the length of a seasonal cycle and referred to as the period. The operator $\mathcal A(t)$ is a second order uniformly elliptic operator that includes intrinsic diffusion, convection and reproduction rates, depending $T$-periodically on $t\in\mathbb R$. The operator $\mathcal B(t)$ determines what happens when the species approaches the boundary $\partial\Omega$ of the habitat and is also a $T$-periodic function of $t\in\mathbb R$. The boundary operator is of Dirichlet, Neumann or Robin type, modelling hostile, no-flux or partly permeable boundaries of the habitat. We also admit mixed boundary conditions with different types of boundary conditions on different parts of $\partial\Omega$. The parameter $\mu\in\mathbb R$ adjusts the reproduction rate. The non-linear term $b(x,t)g(x,t,u)$ is a rate that limits the growth depending on the size of the population. We assume that the growth rate is reduced as the population density increases, which means that the function $u\mapsto g(x,t,u)$ is strictly increasing and that $b(x,t)\geq 0$. The size of $b\in L^\infty(\Omega\times[0,T])$ is a weight that determines how favourable the conditions in the habitat are at any place and time $(x,t)$. Ideal living conditions are given in the region where $b(x,t)=0$, where unrestricted growth is possible. One of the main features in this paper is to deal with the natural assumption that such places change throughout a seasonal cycle. Large values of $b(x,t)$ model a hostile environment.

As demonstrated in \cite{hess:91:ppb}, many properties of \eqref{eq:pp-logistic} are determined by the spectrum of the periodic-parabolic eigenvalue problem
\begin{equation}
  \label{eq:pp-evp}
  \begin{aligned}
    \frac{\partial u}{\partial t}+\mathcal A(t)u+\gamma b(x,t)u & =\mu u &  & \text{in $\Omega\times(0,T)$,} \\ \mathcal B(t)u & =0 &  & \text{on $\partial\Omega\times[0,T]$,} \\ u(\cdot\,,0) & =u(\cdot\,,T) &  & \text{in $\Omega$,}
  \end{aligned}
\end{equation}
where $\gamma\geq 0$ is a parameter. We call an eigenvalue $\mu_1$ a \emph{principal eigenvalue} of \eqref{eq:pp-evp} if it has a positive eigenfunction. A positive eigenfunction of norm one is called a \emph{principal eigenfunction}. It turns out that there is exactly one principal eigenvalue and eigenfunction. We will consider the principal eigenvalue as a function of the zero order term $\gamma b$ and write $\mu_1=\mu_1(\gamma b)$. As shown in \cite[Section~4]{daners:15:ppe} such a periodic-parabolic eigenvalues and eigenfunctions exist. The principal eigenvalue turns out to be strictly increasing and hence
\begin{equation}
  \label{eq:pe-limit}
  \mu^*(b):=\lim_{\gamma\to\infty}\mu_1(\gamma b)\in(-\infty,\infty]
\end{equation}
exists. Under precise assumptions to be specified in Section~\ref{sec:assumptions} and Section~\ref{sec:logistic} we will prove the following result.
\begin{theorem}
  \label{thm:main}
  The equation~\eqref{eq:pp-logistic} has a non-trivial positive weak solution if and only if $\mu_1(0)<\mu<\mu^*(b)$. In that case this solution is unique and linearly stable.
\end{theorem}
A version of the above theorem with $\mu^*(b)<\infty$ appears for the first time in Du and Peng \cite{du:12:ple}, where $b$ had very special spacial and temporal degeneracies, and the linear part of the equation was autonomous. We remove all these restrictions.

As a corollary we recover and generalise a result due to Hess \cite[Section~28]{hess:91:ppb}, but allow $b$ to have degeneracies. It is about the existence and uniqueness of a non-trivial positive solution of
\begin{equation}
  \label{eq:ppp-logistic}
  \begin{aligned}
    \frac{\partial u}{\partial t}+\mathcal A(t)u & =m u-b(x,t)g(x,t,u)u &  & \text{in $\Omega\times(0,T)$,} \\ \mathcal B(t)u & =0 &  & \text{on $\partial\Omega\times[0,T]$,} \\ u(\cdot\,,0) & =u(\cdot\,,T) &  & \text{in $\Omega$,}
  \end{aligned}
\end{equation}
where $m\in L^\infty(\Omega\times(0,T))$ possibly changes sign. It is an immediate consequence of \eqref{eq:pp-logistic} with $\mathcal A(t)$ replaced by $\mathcal A(t)-m(t)$ and $\mu=0$. One can replace $m$ by $\lambda m$ with $\lambda\in\mathbb R$ a parameter, and recover some of the results in \cite{aleja:21:wpp,aleja:23:cep}, but under much weaker assumptions on the regularity of the coefficients and the domains.
\begin{corollary}
  \label{cor:main}
  Under the above assumptions, the problem \eqref{eq:ppp-logistic} has a non-trivial solution if and only if $\mu_1(-m)<0<\mu^*(m,b)$, where
  \begin{equation}
    \label{eq:pep-limit}
    \mu^*(m,b):=\lim_{\gamma\to\infty}\mu_1(\gamma b-m).
  \end{equation}
  In that case this solution is unique and linearly stable.
\end{corollary}

Part of this material is contained in the PhD Thesis \cite{zeaiter:24:psg} by the second author. Related results for linear equations or systems also appear in \cite{alvarez-caudevilla:20:all,alvarez:14:qac}.

\section{Precise assumptions}
\label{sec:assumptions}
We will be working with non-autonomous boundary value problems in divergence form. We assume that $\Omega\subseteq\mathbb R^N$ is a bounded domain, that $T>0$ is fixed and that $\mathcal A(t)$ has the form
\begin{equation}
  \label{eq:A}
  \mathcal A(t)u:=-\sum_{k=1}^N\frac{\partial}{\partial x_k}\Bigl( \sum_{j=1}^N a_{jk}\frac{\partial u}{\partial x_j}+a_ju\Bigr) +\sum_{k=1}^nb_k\frac{\partial u}{\partial x_k}+c_0u,
\end{equation}
where $a_{jk},a_j,b_k,c_0\in L^\infty(\Omega\times \mathbb R)$ are real valued and $T$-periodic in $t\in\mathbb R$. We also assume that $\mathcal A(t)$ is uniformly strongly elliptic, that is, there exists $\alpha>0$ such that
\begin{equation}
  \label{eq:Aelliptic}
  \alpha |\xi|^2\leq\sum_{k=1}^N\sum_{j=1}^N a_{jk}(x,t)\xi_j\xi_k
\end{equation}
for all $\xi\in\mathbb R^N$ and almost all $(x,t)\in\Omega\times\mathbb R$. We assume that $\partial\Omega=\Gamma_0\cup\Gamma_1$ is a disjoint union with $\Gamma_0$. We consider the boundary operator
\begin{equation}
  \label{eq:B}
  \mathcal B(t)u:=
  \begin{cases}
    u_{\Gamma_0}
                                                                                                   & \text{on $\Gamma_0$ (Dirichlet b.c.)}      \\
    \sum_{k=1}^N\Bigl(\sum_{j=1}^N a_{jk}\frac{\partial u}{\partial x_j}+a_ju\Bigr)\nu_k+\beta_0 u & \text{on $\Gamma_1$ (Neumann/Robin b.c.).}
  \end{cases}
\end{equation}
Here $\nu=(\nu_1,\dots,\nu_N)$ is the outer unit normal on $\Gamma_1$, and $\beta_0\in L^\infty(\Gamma_1\times\mathbb R)$ $T$-periodic in $t\in\mathbb R$. Note that we do not assume that the $\Gamma_k$ are open and closed in $\partial\Omega$ and we generally make no restrictions on the sign of $\beta_0$. For simplicity we assume that a neighbourhood of $\Gamma_1$ in $\partial\Omega$ is Lipschitz. In that case we can in particular assume without loss of generality that $\beta_0\geq 0$. According to \cite[Section~3]{daners:09:ipg} this can be achieved by rewriting the pair of operators $(\mathcal A(t),\mathcal B(t))$ in an equivalent form. Hence in what follows we always assume without loss of generality that $\beta_0\geq 0$.

The boundary conditions \eqref{eq:B} as well as the elliptic operator \eqref{eq:A} are to be interpreted in a weak form. To do so we introduce the bilinear form associated with $(\mathcal A(t),\mathcal B(t))$. We define
\begin{multline}
  \label{eq:a0-form}
  \aaa_0(t,u,v):=\int_\Omega\sum_{k=1}^N\Bigl(\sum_{j=1}^Na_{jk}(x,t)\frac{\partial u}{\partial x_j}+a_j(x,t)\Bigr)\frac{\partial v}{\partial x_k}\,dx\\ +\int_\Omega\Bigl(\sum_{j=1}^N b_k(x,t)\frac{\partial u}{\partial x_k}+c_0(x,t)u\Bigr)v\,dx.
\end{multline}
for all $u,v\in H^1(\Omega)$. To incorporate the boundary conditions we introduce the space
\begin{equation}
  \label{eq:V_B}
  H_{\mathcal B}^1(\Omega):=\overline{\bigl\{u\in H^1(\Omega)\colon\text{$u=0$ in a neighbourhood of $\Gamma_0$}\bigr\}},
\end{equation}
where the closure is taken in the Sobolev space $H^1(\Omega)$. We define
\begin{equation}
  \label{eq:a-form}
  \aaa(t,u,v):=\aaa_0(t,u,v)+\int_{\Gamma_1}\beta_0(x,t)uv\,d\sigma,
\end{equation}
for all $u,v\in H_{\mathcal B}^1(\Omega)$, where the integral over $\Gamma_1$ is taken with respect to the ($N-1$)-dimensional Hausdorff measure restricted to $\Gamma_1$. The following proposition collects some standard properties of the form $\aaa$, see for instance \cite[XVIII Section 4.4]{dautray:92:man5}. We will denote by $u^{+}:=\max\left\{u,0\right\}$ and $u^{-}:=\max\left\{-u,0\right\}$ the positive and negative parts of $u\in V$.

\begin{proposition}
  \label{prop:form-properties}
  Let the above assumptions be satisfied and set $V:=H_{\mathcal B}^1(\Omega)$ and $H:=L^{2}(\Omega)$. Then the following assertions are true:
  \begin{enumerate}[\upshape (i)]
  \item $[0,T]\to\mathbb R$, $t\mapsto \aaa(t,u,v)$ is measurable for all $u,v\in V$.
  \item $V\times V\to\mathbb R$, $(u,v)\mapsto \aaa(t,u,v)$ is bilinear and there exist $M\geq 0$ such that
    \begin{equation}
      \label{eq:a-bounded}
      |\aaa(t,u,v)|\leq M\|u\|_V\|v\|_V
    \end{equation}
    for all $t\in[0,T]$ and all $v\in V$.
  \item There exists $\omega_0\in\mathbb R$ such that we can choose $\alpha>0$ so that,
    \begin{equation}
      \label{eq:a-coercive}
      \frac{\alpha}{2}\|u\|_V^2
      \leq \aaa(t,u,u)+\omega \|u\|_H^2
    \end{equation}
    for all $\omega\geq\omega_0$ and all $u\in V$.
  \item $\aaa(t,u^+,u^-)=0$ for all $t\in[0,T]$ and all $u\in V$.
  \end{enumerate}
\end{proposition}
It follows from the above proposition that there exist operators $A(t)\in\mathcal L(V,V')$, $t\in[0,T]$, such that $\left\langle A(t)u,v\right\rangle=\aaa(t;u,v)$ for all $u,v\in V$ and that
\begin{equation}
  \label{eq:A-bounded}
  \|A(t)\|_{\mathcal L(V,V')}\leq M
\end{equation}
for all $t\in [0,T]$. Hence, given $s\in[0,T)$, $u_0\in H$ and $f\in L^2((s,T),V')$ it makes sense to consider the linear initial value problem
\begin{equation}
  \label{eq:linear-abstract-equation}
  \begin{aligned}
    \dot{u}+A(t)u & =f(t) &  & t\in(s,T], \\
    u(s)          & =u_0.                 \\
  \end{aligned}
\end{equation}
We call $u$ a \emph{solution} of \eqref{eq:linear-abstract-equation} if
\begin{equation}
  \label{eq:W}
  u\in W(s,T;V,V'):=\bigl\{u\in L^2((s,T),V)\colon \dot u\in L^2((s,T),V')\bigr\}
\end{equation}
and it satisfies \eqref{eq:linear-abstract-equation}. The space $W(s,T;V,V')$ is a Hilbert space with norm
\begin{equation*}
  \|u\|_W
  :=\Bigl(\int_s^T\|u(t)\|_V^2\,dt+\int_s^T\|\dot u(t)\|_{V'}^2\,dt\Bigr)^{1/2}.
\end{equation*}
It is well known that
\begin{equation}
  \label{eq:W-to-C}
  W(s,T;V,V')\hookrightarrow C([s,T],H),
\end{equation}
see for instance \cite[XVIII Section 1.2]{dautray:92:man5}. Hence, the initial condition $u(s)=v\in H$ makes sense and $u\in C([s,T],H)$. Moreover, as $V\hookrightarrow H$ is compact, the embedding
\begin{equation}
  \label{eq:W-to-L-compact}
  W(s,T;V,V')\hookrightarrow L^2([s,T],H),
\end{equation}
is compact as well, see \cite[Theorem~1.5.1]{lions:69:qmr}. There are several ways to characterise solutions to \eqref{eq:linear-abstract-equation}. In particular $u\in L^2((s,T),V)$ is a solution of \eqref{eq:linear-abstract-equation} if and only if
\begin{equation}
  \label{eq:linear-weak-solution}
  \begin{aligned}
    -\langle u(s),\dot v(s)\rangle
     & +\int_s^T\langle u(\tau),\dot v(\tau)\rangle\,d\tau
    +\int_s^T\aaa(\tau,u(\tau),v(\tau))\,d\tau             \\
     & =\int_s^T\langle f(\tau),v(\tau)\rangle\,d\tau
  \end{aligned}
\end{equation}
for all $v\in W(s,T;V,V')$ with $v(T)=0$ or a dense subset thereof, see for instance \cite{dautray:92:man5}. It can be proved either by Galerkin approximation or by Lion's generalisation of the Lax-Milgram theorem that \eqref{eq:linear-abstract-equation} has a unique solution in $W(s,T;V,V')$, see \cite{dautray:92:man5}, \cite[Section~IV.1]{lions:61:edo} or \cite[Section~III.2]{showalter:97:mob}. We will make use of the following a priori estimates for the solutions. The important feature for us is the independence on a positive potential.
\begin{proposition}[A priori estimates]
  \label{prop:a-priori-estimates}
  Suppose that the above assumptions are satisfied and that $m\in L^\infty((0,T),L^\infty(\Omega))$ with $m\geq 0$. Let $s\in[0,T)$, $u_0\in H$ and $f\in L^2((s,T),V')$. Let $u\in W(s,T;V,V')$ be a solution of
  \begin{equation}
    \label{eq:pp-abstract-perturbed}
    \begin{aligned}
      \dot u+A(t)u+m(t)u & =f(t) &  & t\in(s,T] \\
      u(s)               & =u_0  &  &
    \end{aligned}
  \end{equation}
  Then for any $\omega\geq \omega_0$ and $t\in[s,T]$ we have that
  \begin{equation}
    \label{eq:a-priori-estimate}
    \begin{split}
      \alpha\int_s^t\|e^{-\omega(\tau-s)}u(\tau)\|_V^2\,d\tau
       & +\|e^{-\omega(t-s)}u(t)\|_H^2                                                       \\
       & \leq\|u(s)\|_H^2+\frac{1}{\alpha}\int_s^t\|e^{-\omega(\tau-s)}f(\tau)\|_{V'}\,d\tau
    \end{split}
  \end{equation}
  and
  \begin{equation}
    \label{eq:a-priori-estimate-derivative}
    \begin{split}
      \int_s^t\|e^{-\omega(\tau-s)}\dot u(\tau)\|_{V'}^2\,d\tau
       & \leq 2(M+\|m\|_\infty)^2\int_s^t\|e^{-\omega(\tau-s)}u(\tau)\|_{V}^2\,d\tau \\
       & \qquad+2\int_s^t\|e^{-\omega(\tau-s)}f(\tau)\|_{V'}^2\,d\tau
    \end{split}
  \end{equation}
  where $M$ is from \eqref{eq:a-bounded}.
\end{proposition}
\begin{proof}
  First note that $w(t):=e^{-\omega(t-s)}u(t)$ satisfies the equation
  \begin{equation*}
    \begin{aligned}
      \dot w+A(t)w+\omega w+m(t)w & =e^{-\omega(t-s)}f(t) &  & t\in(s,T] \\
      w(s)                        & =u_0.                 &  &
    \end{aligned}
  \end{equation*}
  By the integration by parts formula for functions in $W(s,T;V,V')$ we have that
  \begin{equation}
    \label{eq:integration-by-parts}
    \frac{1}{2}\bigl(\|w(t)\|_H^2-\|w(s)\|_H^2\bigr)
    =\frac{1}{2}\int_s^t\frac{d}{d\tau}\|w(\tau)\|_H^2\,d\tau
    =\int_0^T\langle\dot w(\tau),w(\tau)\rangle\,d\tau,
  \end{equation}
  see for instance \cite[Theorem~XVIII.1.2]{dautray:92:man5}. Hence it follows from \eqref{eq:a-coercive} and the positivity of $m$ that
  \begin{align*}
    \alpha\int_s^t
     & \|w(\tau)\|_V^2\,d\tau+\frac{1}{2}\|w(t)\|_H^2                                                                                                    \\
     & \leq \frac{1}{2}\|w(s)\|^2+\int_s^t\langle\dot w(\tau),w(\tau)\rangle+\aaa(\tau,w(\tau),w(\tau))+\omega\|w(\tau)\|_H^2\,d\tau                     \\
     & =\frac{1}{2}\|w(s)\|_H^2+\int_s^t\langle\dot w(\tau)+A(\tau)w(\tau)+\omega w(\tau),w(\tau)\rangle\,d\tau                                          \\
     & =\frac{1}{2}\|w(s)\|_H^2+\int_s^te^{-\omega(\tau-s)}\langle f(\tau),w(\tau)\rangle\,d\tau-\int_s^t\langle m(\tau)w(\tau), w(\tau)\rangle\,d\tau   \\
     & \leq\frac{1}{2}\|w(s)\|_H^2+\int_s^t\|e^{-\omega(\tau-s)}f(\tau)\|_{V'}\|w(\tau)\|_V\,d\tau                                                       \\
     & \leq\frac{1}{2}\|w(s)\|_H^2+\frac{1}{2\alpha}\int_s^t\|e^{-\omega(\tau-s)}f(\tau)\|_{V'}^2\,d\tau+\frac{\alpha}{2}\int_s^t\|w(\tau)\|_V^2\,d\tau.
  \end{align*}
  It follows that
  \begin{equation*}
    \alpha\int_s^t\|w(\tau)\|_V^2\,d\tau+\|w(t)\|_H^2
    \leq\|w(s)\|_H^2+\frac{1}{\alpha}\int_s^t\|e^{-\omega(\tau-s)}f(\tau)\|_{V'}^2\,d\tau.
  \end{equation*}
  Taking into account that $w(t)=u(t)e^{-\omega(t-s)}$ we obtain \eqref{eq:a-priori-estimate}. Now consider the derivative $\dot u$. If $v\in V$, then by \eqref{eq:pp-abstract-perturbed}
  \begin{align*}
    |\langle\dot u(\tau),v\rangle
     & =\bigl|-\aaa(\tau,u(\tau),v)-\langle m(\tau)u(\tau),v\rangle+\langle f(\tau),v\rangle\bigr| \\
     & \leq \bigl(M\|u(\tau)\|_V+\|m\|_\infty\|u(\tau)\|_V+\|f(\tau)\|_{V'}\bigr)\|v\|_V
  \end{align*}
  Here we cannot omit the term with $m(\tau)$ since it does not necessarily have a positive sign. By the definition of the dual norm we see that
  \begin{equation*}
    \|\dot u(\tau)\|_{V'}\leq M\|u(\tau)\|_V+\|m\|_\infty\|u(\tau)\|_V+\|f(\tau)\|_{V'}.
  \end{equation*}
  It follows that
  \begin{equation*}
    \|e^{-\omega(\tau-s)}\dot u(\tau)\|_{V'}^2
    \leq 2(M+\|m\|_\infty)^2\|e^{-\omega(\tau-s)}u(\tau)\|_{V}^2+2\|e^{-\omega(\tau-s)}f(\tau)\|_{V'}^2
  \end{equation*}
  Now \eqref{eq:a-priori-estimate-derivative} follows by integration over $[s,t]$.
\end{proof}
\begin{remark}
  \label{rem:W-apriori-estimate}
  Combining \eqref{eq:a-priori-estimate} and \eqref{eq:a-priori-estimate-derivative} we can deduce that any solution of \eqref{eq:linear-abstract-equation} satisfies an estimate of the form
  \begin{equation}
    \label{eq:linear-solution-estimate}
    \|u\|_{W(s,T;V,V')}\leq C\Bigl(\|u_0\|_H^2+\int_s^T\|f(t)\|_{V'}^2\,dt\Bigr)^{1/2}
  \end{equation}
  with $C$ independent of $f$, $u_0$ and $s\in[0,T)$, but dependent on $\|m\|_\infty$. In particular, the solution to \eqref{eq:linear-abstract-equation} is unique and continuously depends on $u_0$ and $f$.
\end{remark}
As a consequence of Proposition~\ref{prop:form-properties}(iv) for any positive initial condition, $v\geq 0$, and inhomogeneity, $f\geq 0$, we have that the solution $u$ of \eqref{eq:linear-abstract-equation} satisfies $u(t)\geq 0$ for all $t\in[s,T]$, see for instance \cite[Proposition~3.1]{arendt:14:ics}. The above iis collected in the following abstract existence theorem.
\begin{theorem}
  \label{thm:linear-existence}
  Under the above assumptions the equation \eqref{eq:linear-abstract-equation} has a unique solution $u\in W(s,T;V,V')$ for every $v\in H$ and $f\in L^2((s,T),V')$. That solution satisfies the a priori estimate \eqref{eq:linear-solution-estimate} with constant depending only on $\alpha$, $\omega$, $M$ and $T$. If $f,u_0\geq 0$, then the solution $u\geq 0$.
\end{theorem}
The following perturbation result is crucial for our treatment of the logistic equation.
\begin{theorem}[perturbation result]
  \label{thm:linear-perturbation}
  Suppose that $m_n,m\in L^\infty((0,T),L^\infty(\Omega))$, that $u_{0n},u_0\in H$ and $f_n,f\in L^2((0,T),V)$ with $m_n\stackrel{*}{\rightharpoonup}m$ weak$^*$ in $L^\infty((0,T),L^\infty(\Omega))$, $u_{0n}\rightharpoonup u_0$ weakly in $H$ and $f_n\rightharpoonup f$ weakly in $L^2((0,T),V')$. Let $s\in[0,T)$ and let $u_n\in W(s,T;V,V')$ be the solution of
        \begin{equation}
          \label{eq:pp-abstract-perturbed-n}
          \begin{aligned}
            \dot u_n+A(t)u+m_n(t)u & =f_n(t)  &  & t\in(s,T], \\
            u(s)                   & =u_{0n}. &  &
          \end{aligned}
        \end{equation}
        Let $u$ be the solution of \eqref{eq:pp-abstract-perturbed}. Then $u_n\rightharpoonup u$ weakly in $W(s,T;V,V')$ and strongly in $L^2((s,T),H)$ as $n\to\infty$. If $f_n\to f$ strongly in $L^2((s,T),V')$, then $u_n(t)\to u(t)$ in $H$ for all $t\in (s,T]$ and $u_n\to u$ in $W(s+\delta,T;V,V')$ for all $\delta>0$ with $s+\delta<T$.
\end{theorem}
\begin{proof}
  We first note that due to the weak$^*$ convergence there exists $C>0$ such that $\|m_n\|_\infty\leq C$ for all $n\in\mathbb N$. Hence it follows from Proposition~\ref{prop:a-priori-estimates} that $(u_n)_{n\in\mathbb N}$ is a bounded sequence in $W(s,T;V,V')$. Hence there exists a subsequence $(u_{n_k})_{k\in\mathbb N}$ such that $u_{n_k}\rightharpoonup u$ weakly in $W(s,T;V,V')$ as $k\to\infty$ for some $u\in W(s,T;V,V')$. By the compact embedding \eqref{eq:W-to-L-compact} we also have strong convergence in $L^2((s,T), H)$. If $v\in W(s,T;V,V')$ with $v(T)=0$, then
  \begin{equation*}
    \begin{aligned}
      -\langle u_{n_k}(s),\dot v(s)\rangle
       & +\int_s^T\langle u_{n_k}(\tau),\dot v(\tau)\rangle\,d\tau
      +\int_s^T\aaa(\tau,u_{n_k}(\tau),v(\tau))\,d\tau                     \\
       & +\int_s^T\langle m_{n_k}(\tau)u_{n_k}(\tau),v(\tau)\rangle\,d\tau
      =\int_s^T\langle f_{n_k}(\tau),v(\tau)\rangle\,d\tau
    \end{aligned}
  \end{equation*}
  for all $k\in\mathbb N$. Letting $k\to\infty$ we see that $u$ is a solution of~\eqref{eq:pp-abstract-perturbed}. By the uniqueness of solutions the full sequence converges. Assume now that $f_n\to f$ strongly. Fix $t\in(s,T]$. Since $u_n\to u$ in $L^2((s,T),H)$ and thus has a subsequence that is convergent almost everywhere there exist $s_0\in (s,t)$ and a subsequence $(u_{n_k})_{k\in\mathbb N}$ such that $u_{n_k}(s_0)\to u(s_0)$ in $H$ as $k\to\infty$. By the integration by parts formula similar to \eqref{eq:integration-by-parts} we have that
  \begin{equation}
    \label{eq:un-estimate}
    \begin{split}
      \alpha & \int_{s_0}^t\|u_n(\tau)-u(\tau)\|_V^2\,d\tau+\frac{1}{2}\|u_n(t)-u(t)\|_H^2                                            \\
             & \leq \frac{1}{2}\|u_n(s_0)-u(s_0)\|_H^2+\int_{s_0}^t\langle\dot u_n(\tau)-\dot u(\tau),u_n(\tau)-u(\tau)\rangle\,d\tau \\
             & \qquad+\int_{s_0}^t\aaa(\tau,u_n(\tau)-u(\tau),u_n(\tau)-u(\tau))\,d\tau                                               \\
             & \qquad+\omega\int_{s_0}^t\|u_n(\tau)-u(\tau)\|_H^2\,d\tau                                                              \\
             & =\int_{s_0}^t\langle f_n(\tau)-f(\tau),u_n(\tau)-u(\tau)\rangle\,d\tau
      +\omega\int_{s_0}^t\|u_n(\tau)-u(\tau)\|_H^2\,d\tau                                                                             \\
             & \qquad-\int_{s_0}^t\langle m_n(\tau)u_n(\tau)-m(\tau)u(\tau),u_n(\tau)-u(\tau)\rangle\,d\tau.
    \end{split}
  \end{equation}
  We know that $u_n-u\to 0$ in $L^2((s,T),H)$ and $m_nu_n-mu\rightharpoonup 0$ weakly in $L^2((s,T),H)$. By assumption $f_n-f\to 0$ in $L^2((s,T),V')$ and $u_{n_k}(s_0)-u(s_0)\to 0$ in $H$. Hence \eqref{eq:un-estimate} shows that
  \begin{equation*}
    \lim_{k\to\infty}\|u_{n_k}(t)-u(t)\|_H^2=0.
  \end{equation*}
  By the uniqueness of the limit, the full sequence converges, which in particular means that $u_n(t)\to u(t)$ in $H$ for all $t\in (s,T]$. Hence, choosing any $s_0\in(s,T]$ it follows that $u_n\to u$ in $L^2((s_0,T),V)$ as $n\to\infty$. Next we consider the convergence of $\dot u_n$. Given $v\in V$ we have that
  \begin{align*}
    |\langle\dot u_n & (\tau)-\dot u(\tau),v\rangle|                                                                               \\
                     & =\bigl|\langle A(\tau)(u_n(\tau)-u(\tau))
    +(m_n(\tau)u_n(\tau)-m(\tau)u(\tau))-(f_n(\tau)-f(\tau)),v\rangle\bigr|                                                        \\
                     & \leq \bigl(M\|u_n(\tau)-u(\tau)\|_V+\|m_n(\tau)u_n(\tau)-m(\tau)u(\tau)\|_{v'}+\|f_n-f\|_{V'}\bigr)\|v\|_V.
  \end{align*}
  Hence by definition of the dual norm
  \begin{equation}
    \label{eq:un-derivative-estimate}
    \begin{split}
      \|\dot u_n(\tau) & -\dot u(\tau)\|_{V'}                                                                    \\
                       & \leq M\|u_n(\tau)-u(\tau)\|_V+\|m_n(\tau)u_n(\tau)-m(\tau)u(\tau)\|_{v'}+\|f_n-f\|_{V'}
    \end{split}
  \end{equation}
  for all $\tau\in(s,T]$. We know from the first part of the proof that $m_n(\tau)u_n(\tau)\rightharpoonup m(\tau)u(\tau)$ weakly in $L^2((s,T),H)$. If we can show that $m_nu_n\to mu$ in $L^2((s,T),V')$, then it follows from \eqref{eq:un-derivative-estimate} that $\dot u_n\to \dot u$ in $L^2((s_0,T),V')$ as $n\to\infty$. Since
  \begin{equation*}
    \|m_n u_n-mu\|_{V'}\leq\|m_n\|_\infty\|u_n-u\|_{V'}+\|(m_n-m)u\|_{V'}
  \end{equation*}
  it is sufficient to show that $m_nu\to mu$ in $L^2((s,T),V')$. To do so we use the compactness criterion \cite[Theorem~1]{vannerven:14:clb} for vector valued $L^p$-spaces. We know that $m_nu\rightharpoonup mu$ weakly in $L^2((s,T),H)$. Moreover, given that $\|m_n\|\leq C$ for all $n\in\mathbb N$, we have
  \begin{equation*}
    |\langle m_n(\tau)u(\tau)-m(\tau)u(\tau),v\rangle|\leq 2C\|u(\tau)\|_H\|v\|_H
  \end{equation*}
  for all $v\in V$ and $\tau\in (s,T)$. Hence by the dominated convergence theorem
  \begin{equation*}
    \lim_{n\to\infty}\int_s^T|\langle m_n(\tau)u(\tau)-m(\tau)u(\tau),v\rangle|^2\,d\tau =  0.
  \end{equation*}
  The convergence means that the family $(m_nu)_{n\in\mathbb N}$ is scalarly compact. Fix $\varepsilon>0$ and let
  \begin{equation*}
    J_r:=\{\tau\in(s,T)\colon C\|u(\tau)\|_H\geq r\}.
  \end{equation*}
  We know that for every $\varepsilon>0$ there exists $r>0$ such that
  \begin{equation*}
    \int_{J_r}\|Cu(\tau)\|_H^2\,d\tau<\varepsilon.
  \end{equation*}
  It implies that
  \begin{equation*}
    \int_{J_r}\|m_n(\tau)u(\tau)\|_H^2\,d\tau<\varepsilon
  \end{equation*}
  for all $n\in\mathbb N$. Hence, as $H\hookrightarrow V'$, the family $(m_nu)_{n\in\mathbb N}$ is equi-integrable in $L^2((s,T),V')$. We can also choose $r>0$ such that $(s,T)\setminus J_r$ has measure less than $\varepsilon$. Also,
  \begin{equation*}
    m_n(\tau)u(\tau)\in B_r:=\{v\in H\colon \|v\|_H\leq r/c\}
  \end{equation*}
  for all $n\in\mathbb N$. By the compact embedding $H\hookrightarrow V$ it follows that $B_r\subseteq V'$ is relatively compact in $V'$. It means that $m_n(\tau)u(\tau)$ is in the relatively compact set $B_r$ for all $\tau\in J_r$ and thus the family $(m_nu)_{n\in\mathbb N}$ is uniformly tight as defined in \cite{vannerven:14:clb}. It follows that $(m_nu)_{n\in\mathbb N}$ is relatively compact in $L^2((s,T),V')$ as claimed.
\end{proof}

\section{Evolution systems and their properties}
\label{sec:evolution-systems}
As a special case of \eqref{eq:linear-abstract-equation} we can look at solutions to the homogeneous problem
\begin{equation}
  \label{eq:linear-evolution}
  \begin{aligned}
    \dot{u}+A(t)u & =0\qquad t\in(s,T] \\
    u(s)          & =v
  \end{aligned}
\end{equation}
for $v\in H$. The unique solvability of the linear homogeneous problem \eqref{eq:linear-evolution} stated in Theorem~\ref{thm:linear-existence} allows us to define
\begin{equation*}
  U(t,s)v:=u(t)
\end{equation*}
where $u$ is the unique solution of \eqref{eq:linear-evolution}. By the linearity, $U(t,s)$ is a linear operator on $H$. We call the family $(U(t,s))_{(t,s)\in\Delta_T}$ the \emph{evolution system} on $H$ associated with $(A(t))_{t\in[0,T]}$, where
\begin{equation}
  \label{eq:Delta}
  \Delta_T:=\{(t,s)\colon 0\leq s\leq t\leq T\}.
\end{equation}
We furthermore set
\begin{equation}
  \label{eq:dDelta}
  \dot\Delta_T:=\{(t,s)\colon 0\leq s<t\leq T\}.
\end{equation}
We denote by $\mathcal L_s(H)$ is the space of bounded linear operators on $H$ with the strong operator topology, that is, the topology of pointwise convergence. We next collect the properties of the evolution system we need, see also \cite[Chapter~3]{zeaiter:24:psg}.

\begin{proposition}[Properties of evolution system]
  \label{prop:evolution-system}
  The evolution system has the following properties:
  \begin{enumerate}[\normalfont (U1)]
  \item $U(t,t)=I$ for all $t\in[0,T]$;
  \item $U(t,s)=U(t,\tau)U(\tau,s)$ for all $0\leq s\leq\tau\leq t\leq T$;
  \item $U(\cdot,s)\in C([s,T],\mathcal L_s(H))$ for all $s\in[0,T)$;
  \item $U(t,\cdot)\in C\left([0,t],\mathcal L_s(H)\right)$ for all $t\in(0,T]$;
  \item $U(t,s)\in\mathcal L(H,BC(\Omega))$ and there exists $C>0$ with
    \begin{equation}
      \label{eq:2-infty-estimate}
      \|U(t,s)\|_{\mathcal L(L^2,L^\infty)}\leq C(t-s)^{-N/4}
    \end{equation}
    for all $(t,s)\in\dot\Delta_T$;
  \item $U(t,s)\in\mathcal L(L^\infty(\Omega))$ for all $(t,s)\in\Delta_T$ and $\sup_{(t,s)\in\Delta_T}\|U(t,s)\|_{\mathcal L(L^\infty(\Omega))}<\infty$;
  \item If $v>0$, then $[U(t,s)v](x)>0$ for all $(t,s)\in\dot\Delta_T$ and all $x\in\Omega$;
  \item $U(t,s)\in\mathcal L(H)$ is compact for all $(t,s)\in\dot\Delta_T$.
  \end{enumerate}
  Moreover, if $f\in L^2((s,T),H)$ and $v\in H$, then the solution $u$ of \eqref{eq:linear-abstract-equation} can be represented in the form
  \begin{equation}
    \label{eq:variation-of-constants}
    u(t)=U(t,s)v+\int_s^tU(t,\tau)f(\tau)\,d\tau.
  \end{equation}
\end{proposition}
\begin{proof}
  Properties (U1)--(U3) are a direct consequence of Theorem~\ref{thm:linear-existence} and \eqref{eq:W-to-C}. Property (U4) follows by looking at the dual problem as in \cite[Theorem~2.6]{daners:00:hke}. It follows from heat kernel estimates, for instance in \cite[Corollary~7.2]{daners:00:hke}, that $U(t,s)\in\mathcal L(H,L^\infty(\Omega))$ for all $(t,s)\in\dot\Delta_T$ with the given estimate. The continuity of $U(t,s)v$ as a function of $x\in\Omega$ then follows from standard H\"older estimates for parabolic equations such as those in \cite{aronson:68:nsl}. The boundedness on $L^\infty(\Omega)$ in (U6) follows from \cite[Corollary~7.2]{daners:00:hke}.

  If all coefficients of $\mathcal A(t)$ and $\mathcal B(t)$ and the domain $\Omega$ are smooth, then the parabolic maximum principle implies that $[U(t,s)v](x)>0$ for all $(x,t)\in\Omega\times(s,T]$. Using \cite[Theorem~8.3 and Lemma~8.4]{daners:00:hke} the same statement holds in general. Alternatively one could use the parabolic Harnack inequality from \cite{aronson:67:lbs}. This proves (U7).

  To prove (U8) we note that $L^\infty(\Omega)$ is the principal ideal generated by the constant function with value one in the Banach lattice $H$. By (U5) we have $U(t,s)H\subseteq L^\infty(\Omega)$ for all $(t,s)\in\dot\Delta_T$. As $U(t,s)=U(t,\tau)U(\tau,s)$ whenever $0\leq s<\tau<t\leq T$ it follows from \cite[Theorem~2.2]{daners:17:rds} that $U(t,s)\in\mathcal L(H)$ is compact.

  If we deal with Dirichlet boundary conditions, then \eqref{eq:variation-of-constants} follows from \cite[Theorem~9]{aronson:68:nsl}. In a more abstract setting, using the Galerkin approximation, the formula holds for the finite dimensional approximations, and remains valid by taking a limit.
\end{proof}

We refer to \eqref{eq:variation-of-constants} as the the \emph{variation-of-constants} formula. We next prove a perturbation theorem, weaker than that usually found in the literature. The proof only relies on the properties of the evolution system and is valid in other settings as well, but to keep our exposition simple we refrain from proving a more general version.

\begin{theorem}
  \label{thm:evolution-convergence}
  Suppose that $b_n\in L^\infty([0,T],L^\infty(\Omega))$ is such that $b_n(\tau)\stackrel{*}{\rightharpoonup}b(\tau)$ weak$^*$ in $L^\infty(\Omega)$ for almost all $\tau\in (0,T)$. Let $(U_n(t,s))_{(t,s)\in\Delta_T}$ be the evolution system associated with $(A(t)+b_n(t)))_{t\in[0,T]}$ and let $(U(t,s))_{(t,s)\in\Delta_T}$ be the evolution system associated with $(A(t)+b(t)))_{t\in[0,T]}$. Then $U_n(t,s)\to U(t,s)$ in $\mathcal L(H)$ as $n\to\infty$.
\end{theorem}
\begin{proof}
  Let $0\leq s<t\leq T$. By taking $f_n=0$ it follows from Theorem~\ref{thm:linear-perturbation} that $U(t,s)v_n\to U(t,s)v$ whenever $v_n\rightharpoonup v$ weakly in $H$. By Proposition~\ref{prop:evolution-system} $U(t,s)$ is compact. Hence, by \cite[Proposition~4.1.1]{daners:08:dpl} it follows that $U_n(t,s)\to U(t,s)$ in $\mathcal L(H)$ as $n\to\infty$.
\end{proof}

Using \cite[Proposition~2.2]{arendt:23:see} we obtain as a corollary to Theorem~\ref{thm:evolution-convergence} the following perturbation result on the spectral radius of $U(t,s)$. The result is also related to \cite[Lemma~15.7]{hess:91:ppb}, but uses much less regularity.

\begin{corollary}
  \label{cor:spectral-radius-convergence}
  Under the assumption of Theorem~\ref{thm:evolution-convergence} we have that,
  \begin{equation*}
    \spr(U_{n}(t,s)) \to \spr(U(t,s))
  \end{equation*}
  for all $(t,s)\in\dot\Delta_T$. Moreover, if $v_n$ is the positive eigenfunction of $U_n(t,s)$ with $\|v_n\|_2=1$, then $v_n\to v$ in $H$ and $v$ is the positive eigenfunction of $U(t,s)$.
\end{corollary}

We will use the above perturbation theorem with a specific set of perturbations.

\begin{remark}
  \label{rem:cutoff-convergence}
  Let $b\in L^\infty((0,T),L^\infty(\Omega)$ with $b\geq 0$. For $\delta>0$ define
  \begin{equation*}
    \Omega_{\delta}
    :=\left\{x\in\Omega \colon \dist (x,\partial\Omega)>\delta\right\}
  \end{equation*}
  and let
  \begin{equation*}
    b_{\delta}(x,t):=b(x,t)1_{\Omega_\delta}(x)
  \end{equation*}
  for all $(x,t)\in\Omega\times[0,T]$. Then $b_\delta$ has compact support in $\Omega\times[0,T]$ and $b_\delta\uparrow b$ pointwise. In particuar, if we choose $\delta_n\downarrow 0$ we have that  $b_\delta(t)\stackrel{*}{\rightharpoonup} b(t)$ weak$^*$ in $L^\infty(\Omega)$ as $\delta\to 0^+$ for all $t\in[0,T]$.
\end{remark}
We also need a comparison theorem.
\begin{proposition}
  \label{prop:U-comparison}
  Let $b_1,b_2\in L^\infty((0,T),L^\infty(\Omega))$ with $0\leq b_1<b_2$. For $k=1,2$ let $(U_k(t,s))_{(t,s)\in\Delta_T}$ be the evolution systems associated with $(A(t)+b_k(t))_{t\in[0,T]}$. Then $U_2(t,s)\leq U_1(t,s)$ for all $(t,s)\in\Delta_T$. Moreover, $U_2(T,0)v\ll U_1(T,0)v$ for all $v\in H$ with $v>0$.
\end{proposition}
\begin{proof}
  Let $v\in H$ with $v>0$ and let $u_k(t):=U_k(t,s)v$ for $k=1,2$. Then $u_1,u_2\geq 0$ and we have that
  \begin{equation*}
    \begin{aligned}
      (\dot u_1-\dot u_2)+A(t)(u_1-u_2)+b_1(t)(u_1-u_2)
                    & =(b_2(t)-b_1(t))u_2 &  & t\in(s,T] \\
      u_1(s)-u_2(s) & =0.                 &  &
    \end{aligned}
  \end{equation*}
  As $(b_2-b_1)u_2\geq 0$ it follows from Theorem~\ref{thm:linear-existence} that $u_2(t)\leq u_1(t)$ for all $(t,s)\in\Delta_T$. By the positivity of $U(t,s)$ and since $b_2-b_1\geq 0$ we deduce that $U_2(t,s)\leq U_1(t,s)$. To show that $U_2(T,0)v\ll U_1(T,0)v$ we note that by Proposition~\ref{prop:evolution-system} we have $U(t,\tau)v\gg 0$ for all $\tau\in[0,t)$. By assumption $b_2(\tau)-b_1(\tau)>0$ for $\tau$ in a set of positive measure in $(0,T)$ and hence $[b_2(\tau) - b_1(\tau)]U_1(\tau,s)v>0$. It follows that
  \begin{equation*}
    U_2(T,\tau)\left[b_2(\tau) - b_1(\tau)\right]U_1(\tau,0)v\gg 0
  \end{equation*}
  for $\tau$ in a set of positive measure in $(0,T)$. Hence
  \begin{align*}
    U_1(T,0)v & = U_2(T,0)v + \int_{0}^{T}U_2(T,\tau)\left[b_2(\tau) - b_1(\tau)\right]U_1(\tau,0)v\,d\tau \\
              & \gg U_2(T,0)v
  \end{align*}
  as claimed.
\end{proof}

\section{Periodic-parabolic eigenvalue problems}
\label{sec:periodic-parabolic-eigenvalues}
In this section we keep the notation and assumptions from Section~\ref{sec:assumptions} and \ref{sec:evolution-systems}. We discuss properties of the \emph{periodic-parabolic eigenvalue problem},
\begin{equation}
  \label{eq:pp-evp-abstract}
  \begin{aligned}
    \dot{u} + A(t)u + b(t)u & = \mu u &  & t\in[0,T] \\
    u(0)                    & = u(T)
  \end{aligned}
\end{equation}
in $H$, where $b\in L^{\infty}([0,T],L^{\infty}(\Omega))$. We denote the evolution system associated with $(A(t)+b(t))_{t\in[0,T]}$ by $(U_{b}(t,s))_{(t,s)\in\Delta_T}$.

\begin{definition}
  \label{def:principal-ev}
  We call $\mu$ a \emph{principal} eigenvalue of \eqref{eq:pp-evp-abstract} if there exists $u\in W(0,T;V,V')$ such that $u>0$ and  \eqref{eq:pp-evp-abstract} is satisfied. Then, $u$ is called the \emph{principal} eigenfunction corresponding to the eigenvalue $\mu$. We define the principal eigenvalue and eigenfunction of an operator similarly.
\end{definition}

We make use of the following proposition to reason the existence and uniqueness of a $T$-periodic principal eigenfunction of \eqref{eq:pp-evp-abstract}. This is similar to \cite[Proposition~14.4]{hess:91:ppb}.

\begin{lemma}
  \label{lem:U-ev-relation}
  Let $b\in L^\infty((0,T),L^\infty(\Omega))$ with $b\geq 0$. Then $\lambda\in\mathbb{R}$ is an eigenvalue of $U_{b}(T,0)$ if and only if
  \begin{equation}
    \label{eq:U-ev-relation}
    \mu = -\frac{1}{T}\log(\lambda)
  \end{equation}
  is an eigenvalue of \eqref{eq:pp-evp-abstract}. If $v$ is an eigenvector of $U_b(T,0)$ corresponding to $\lambda$, then
  \begin{equation*}
    u(t) = e^{\mu t}U_b(t,0)v.
  \end{equation*}
  is an eigenfunction of \eqref{eq:pp-evp-abstract} corresponding to $\lambda$. Likewise, if $u$ is an eigenfunction for \eqref{eq:pp-evp-abstract}, then $v:=u(0)$ is an eigenvector of $U_b(T,0)$. Moreover, $\lambda$ is a principal eigenvalue of $U_b(T,0)$ if and only if $\mu$ is a pricipal eigenvalue of \eqref{eq:pp-evp-abstract}
\end{lemma}

Under some further assumptions on $U_{b}(T,0)$ we can guarantee the existence of a unique principal eigenvalue of \eqref{eq:pp-evp-abstract}.

\begin{theorem}
  \label{thm:pp-evp-abstract-existence}
  Let $b,b_1,b_2\in L^\infty((0,T),L^\infty(\Omega))$. Then problem \eqref{eq:pp-evp-abstract} has a unique principal eigenvalue we denote by $\mu_1(b)$ and a unique principal eigenfunction. Moreover, if $\mu$ is another eigenvalue of \eqref{eq:pp-evp-abstract}, then $\mu_1(b)\leq\repart(\mu)$. Finally, if $b_1\leq b_2$, then $\mu_1(b_1)\leq\mu_1(b_2)$ with equality if and only if $b_1=b_2$ almost everywhere.
\end{theorem}
\begin{proof}
  As a result of Proposition~\ref{prop:evolution-system} $U_{b}(T,0)$ is compact by (U8) and irreducible on $H$ by (U7). By de Pagter's theorem
  \begin{equation}
    \label{eq:positive-spr}
    \spr(b):=\spr(U_{b}(T,0))>0,
  \end{equation}
  see \cite[Theorem~3]{depagter:86:ico} or \cite[Theorem~4.2.2]{meyer:91:bla}. Now by an application of the Krein-Rutman theorem we have the existence of a principal eivenvector $v_{b}\geq 0$, see \cite[Theorem~41.2]{zaanen:97:iot}. Since $U_{b}(T,0)$ is positive and irreducible, we have that $v_{b}$ is a quasi-interior point, that is, $v_b\gg 0$, see for instance \cite[Lemma~4.2.9]{meyer:91:bla}. Moreover, by \cite[Corollary~4.2.15]{meyer:91:bla} $\spr(U_b(T,0))$ is algebraically simple and the only eigenvalue of $U_b(T,0)$ having a positive eigenvector. The relationship established in Lemma~\ref{lem:U-ev-relation} gives a unique principal eigenvalue and eigenfunction pair of \eqref{eq:pp-evp-abstract} when suitably normalized. If $\mu\neq\mu_1(b)$ is an eigenvalue of \eqref{eq:pp-evp-abstract}, then we know that $|\lambda|\leq\spr(U_b(T,0))$. Then the corresponding eigenvalue $\mu$ of \eqref{eq:pp-evp-abstract} given by \eqref{eq:U-ev-relation} satisfies $\repart(\mu)\geq\mu_1(b)$.

  We next look at the comparison of eiganvalues. Using the notation from Proposition~\ref{prop:U-comparison} we have $0\leq U_2(T,0)<U_1(T,0)$ if and only if $b_1<b_2$. As $U_1(T,0)$ and $U_2(T,0)$ are compact and irreducible it follows from \cite[Theorem~2.1]{arendt:23:see} that $0<\spr(U_2(T,0))\leq\spr(U_1(T,0))$ with equality if and only if $U_2(T,0)=U_1(T,0)$. By the relationship \eqref{eq:U-ev-relation} we have that $\mu_1(b_1)\leq \mu_1(b_2)$ with equality if and only if $b_1=b_2$ almost everywhere.
\end{proof}

It is beneficial to study properties of \eqref{eq:pp-evp-abstract} for the family of potentials $(\gamma b)_{\gamma\geq0}$. In such a case let us denote the corresponding evolution system by $U_{\gamma}(t,s)$. The following proposition is a slight generalisation of \cite[Theorem~3.1 and~4.2]{daners:15:ppe}.
\begin{proposition}
  \label{prop:es-ev-monotonicity}
  Suppose that $b\in L^\infty((0,T),L^\infty(\Omega))$ with $b>0$.  Denote by $(U_\gamma(t,s))_{(t,s)\in\Delta_T}$ the evolution system associated with $(A(t)+\gamma b(t))_{t\in[0,T]}$. Let $\mu_1(\gamma b)$ be the principal eigenvalue and $0<\varphi_\gamma$ a corresponding eigenfunction of \eqref{eq:pp-evp-abstract}. Then the following assertions are true
  \begin{enumerate}[\normalfont (i)]
  \item If $v\in H$ and $v\geq 0$, then $U_{\gamma}(t,s)v$ is decreasing as a function of $\gamma\geq 0$ for all $(t,s)\in \Delta_T$.
  \item $\mu_{1}(\gamma b)$ is increasing as a function of $\gamma\geq 0$ and
    \begin{equation}
      \label{eq:mu-star}
      \mu^{*}(b) := \lim_{\gamma\to\infty}\mu_{1}(\gamma b)\in(\mu_{1}(0),\infty]
    \end{equation}
    exists.
  \item If $\mu^*(b)<\infty$, then there exists $c>0$ such that
    \begin{equation}
      \label{eq:ev-bounded}
      \|\varphi_{\gamma}\|_{L^{\infty}(\Omega\times[0,T])}\leq c\|\varphi_\gamma(0)\|_2
    \end{equation}
    for all $\gamma\geq 0$.
  \item  $\varphi_\gamma\in C(\Omega\times[0,T])$ for all $\gamma\geq 0$ and $\varphi_\gamma(x,t)>0$ for all $(x,t)\in\Omega\times[0,T]$.
  \item For all $\gamma\geq 0$ and $\omega\geq\max\{0,\omega_0)$ we have
    \begin{equation}
      \label{eq:ev-bounded-V}
      \begin{split}
        \alpha\int_0^T\|\varphi_\gamma(\tau)\|_V^2\,d\tau+\gamma\int_0^
        T & \langle b(\tau)\varphi_\gamma(\tau),\varphi_\gamma(\tau)\rangle\,d\tau    \\
          & \leq (\mu_1(\gamma b)+\omega)\int_0^T\|\varphi_\gamma(\tau)\|_2^2\,d\tau,
      \end{split}
    \end{equation}
    where $\omega_0$ is from \eqref{eq:a-coercive}.
  \end{enumerate}
\end{proposition}
\begin{proof}
  Parts (i) and (ii) follow directly from Proposition~\ref{prop:U-comparison} and the fact that any increasing function has a proper or improper limit. To prove part (iii) note that by (i), Lemma~\ref{lem:U-ev-relation}, Proposition~\ref{prop:evolution-system}, Proposition~\ref{prop:U-comparison} and the $T$-periodicity of $\varphi_\gamma$ we have
  \begin{align*}
    \|\varphi_\gamma(t)\|_\infty
     & =\|\varphi_\gamma(T+t)\|_\infty                                                               \\
     & =e^{\mu_1(\gamma b)(T+t)}\|U_\gamma(t+T,0)\varphi_\gamma(0)\|_\infty                          \\
     & \leq e^{\mu_1(\gamma b)(T+t)}\|U_0(t+T,0)\|_{\mathcal L(L^2,L^\infty)}\|\varphi_\gamma(0)\|_2 \\
     & \leq e^{2|\mu^*(b)|T}CT^{-N/4}\|\varphi_\gamma(0)\|_2
  \end{align*}
  for all $t\in[0,T]$. Hence we can set $c=:e^{2|\mu^*(b)|T}CT^{-N/4}$ to conclude the proof. Part (iv) follows from Proposition~\ref{prop:evolution-system} (U5) and (U7) since
  \begin{equation*}
    \varphi_\gamma(t)
    =\varphi_\gamma(T+t)=e^{\mu_1(\gamma b)(T+t)}U_\gamma(t+T,0)\varphi_\gamma(0)
  \end{equation*}
  for all $t\in[0,T]$. For part (v) we use the integration by parts formula for functions in $W(0,T;V,V')$ to conclude that
  \begin{equation*}
    0=\frac{1}{2}\left(\|\varphi_\gamma(T)\|_2-\|\varphi_\gamma(0)\|_2\right)
    =\int_0^T\frac{d}{d\tau}\|\varphi_\gamma(\tau)\|_2^2\,d\tau
    =\int_0^T\langle\dot\varphi_\gamma(\tau),\varphi_\gamma(\tau)\rangle\,d\tau.
  \end{equation*}
  see for instance \cite[Theorem~XVIII.1.2]{dautray:92:man5}. Hence, using \eqref{eq:a-coercive} and the fact that $\varphi_\gamma$ is an eigenfunction, we obtain
  \begin{equation*}
    \begin{split}
      \alpha\int_0^T\|\varphi_\gamma(\tau)\|_V^2\,d\tau
       & +\gamma\int_0^T\langle  b(\tau)\varphi_\gamma(\tau),\varphi_\gamma(\tau)\rangle\,d\tau \\
       & \leq\int_0^T\aaa(\tau,\varphi_\gamma(\tau),\varphi_\gamma),d\tau
      +\gamma\int_0^T\langle  b(\tau)\varphi_\gamma(\tau),\varphi_\gamma(\tau)\rangle\,d\tau    \\
       & \qquad+\omega\int_0^T\|\varphi_\gamma(\tau)\|_2^2\,d\tau                               \\
       & =(\mu_1(\gamma b)+\omega)\int_0^T\|\varphi_\gamma(\tau)\|_2^2\,d\tau
    \end{split}
  \end{equation*}
  as claimed.
\end{proof}

If the family of potentials $(b_{\gamma})_{\gamma\geq0}$ are of a particular form we can make conclusions on the behaviour of solutions to the equivalent homogeneous problem to \eqref{eq:pp-evp-abstract}. This is a slightly different version of \cite[Theorem~4.2]{daners:15:ppe} in the sense that no explicit assumptions on the zero set of $b$ are made. To formulate the theorem we introduce the set
\begin{equation}
  \label{eq:b-positive}
  Q_b:=\left\{(x,t)\in \Omega\times[0,T]\colon\liminf_{(y,s)\to(x,t)} b(x,t)>0\right\},
\end{equation}

\begin{theorem}
  \label{thm:limit-support}
  Let $b\in L^\infty((0,T),L^\infty(\Omega))$ and set $b_{\gamma}(t):=\gamma b(t)$, $\gamma\geq0$. Let $\mu_1(\gamma b)$ and $\varphi_\gamma$ be the principal eigenvalue and eigenfunction of \eqref{eq:pp-evp-abstract} normalised such that $\|\varphi_\gamma(0)\|_2=1$. If $\mu^*(b)<\infty$, \eqref{eq:mu-star}, then there exists a sequence $(\gamma_k)$ with $\gamma_k\to\infty$ such that
  \begin{equation*}
    \varphi_{\infty}(t):=\lim_{k\to\infty}\varphi_{\gamma_k}(t),
  \end{equation*}
  exists in $H$ and $b(t)\varphi_{\infty}(t)=0$ for almost all $t\in(0,T]$. Moreover, $0<\varphi_\infty\in L^\infty(\Omega\times (0,T))$ and $\varphi_{\gamma_k}\rightharpoonup\varphi_\infty$ weakly in $L^2((0,T),V)$. Finally, $\varphi_\infty\to 0$ locally uniformly in $Q_b$.
\end{theorem}
\begin{proof}
  Let $(U_\gamma(t,s))_{(t,s)\in\Delta_T}$ be the evolution system associated with $(A(t)+\gamma b(t))_{t\geq 0}$. By \cite[Theorem~3.1 and~3.2]{daners:15:ppe} it follows that $U_\gamma(t,s)$ is decreasing as a function of $\gamma$ and that
  \begin{equation*}
    U_\infty(t,s):=\lim_{\gamma\to\infty}U_\gamma(t,s)
  \end{equation*}
  exists in $\mathcal L(H)$ for all $(t,s)\in\dot\Delta_T$. Moreover, $U_\infty(t,s)=U_\infty(t,\tau)U_\infty(\tau,s)$ for all $0\leq s<\tau<t\leq T$. As a limit of compact positive operators $U_\infty(t,s)$ is compact and positive as well. Since $\|\varphi_\gamma(0)\|_2=1$ for all $\gamma\geq 0$ it follows that there exists an increasing sequence $(\gamma_k)$ in $[0,\infty)$ such that $\gamma_k\to\infty$ and $\varphi_{\gamma_k}(0)\rightharpoonup v_\infty$ weakly in $H$ as $k\to\infty$. It follows from \cite[Proposition~4.4.1]{daners:08:dpl} that
          \begin{equation*}
            \lim_{k\to\infty}\varphi_{\gamma_k}(0)
            =\lim_{k\to\infty}e^{\mu_1(\gamma_k b)T}U_{\gamma_k}(T,0)\varphi_{\gamma_k}(0)
            =e^{\mu^*(b)T}U_\infty(T,0)v_\infty
          \end{equation*}
          in $H$. In particular $\varphi_{\gamma_k}(0)\to v_\infty$ in $H$ and thus $\|v_\infty\|_2=1$. We let
          \begin{equation*}
            \varphi_\infty(t):=e^{\mu^*(b)t}U_\infty(t,0)v_\infty.
          \end{equation*}
          for all $t\in(0,T]$ and extend $T$-periodically to $t\in\mathbb R$. As $U_\gamma(t,0)\to U_\infty(t,0)$ in $\mathcal L(H)$ we have that $\varphi_{\gamma_k}(t)\to \varphi_\infty(t)$ in $H$ for all $t\in[0,T]$. Moreover, by \eqref{eq:ev-bounded} it follows that $\varphi_\infty\in L^\infty(\Omega\times (0,T))$ and $\varphi_{\gamma_k}\to\varphi_\infty$ in $L^2((0,T),H)$ as $k\to\infty$. As $\mu^*(b)<$ it follows from \eqref{eq:ev-bounded-V} that there is also weak convergence in $L^2((0,T),V)$. Letting $k\to\infty$ in \eqref{eq:ev-bounded-V} also implies that
  \begin{equation*}
    0=\lim_{k\to\infty}\int_0^T\langle b(\tau)\varphi_{\gamma_k}(\tau),\varphi_{\gamma_k}(\tau)\rangle\,d\tau
    =\int_0^T\langle b(\tau)\varphi_\infty(\tau),\varphi_\infty(\tau)\rangle\,d\tau.
  \end{equation*}
  Hence $b\varphi_\infty=0$ almost everywhere on $\Omega\times (0,T)$. We finally have that
  \begin{equation*}
    \varphi_{\gamma}(t)
    =e^{\mu_1(\gamma b)T}U_{\gamma}(t+T,0)\varphi_{\gamma}(0)
    \leq e^{2|\mu^*(b)|T}T^{-N/4}U_\gamma(t+T,0)1
  \end{equation*}
  for all $\gamma>0$ and all $t\in[0,T]$. We know that $U_\gamma(t+T,0)1=U_\gamma(t,0)1\downarrow U_\infty(t,0)1$.
\end{proof}
We next give a criterion for $\mu^*(b)$ to be finite established originally in \cite{daners:15:ppe}.
\begin{lemma}
  \label{lem:mu-star-finite}
  Suppose that $b\in L^\infty(\mathbb R,L^\infty(\Omega))$ is $T$-periodic. Let
  \begin{equation}
    \label{eq:b-zero-set}
    Q_0:=\interior\{(x,t)\in\Omega\times\mathbb R\colon b=0 \text{ a.e. in a neighbourhood relative to }\Omega\times\mathbb{R}\}.
  \end{equation}
  Assume that there exists a $T$-periodic funcion $\beta\in C(\mathbb R,\Omega)$ such that $(\beta(t),t)\in Q_0$ for all $t\in\mathbb R$. Then $\mu^*(b)<\infty$.
\end{lemma}
\begin{proof}
  As $Q_0$ is open, there exists $r>0$ such that
  \begin{equation*}
    A:=\{(x,t)\in\Omega\times\mathbb R\colon \|x-\beta(t)\|\leq r\}\subseteq Q_0.
  \end{equation*}
  Define $\tilde b(x,t)=0$ for all $(x,t)\in A$ and $\tilde b(x,t)=\|b\|_\infty$ otherwise. Then $b\leq\tilde b$ and thus by Theorem~\ref{thm:pp-evp-abstract-existence} we have that $\mu_1(\gamma b)\leq\mu_1(\gamma\tilde b)$ for all $\gamma>0$. By \cite[Theorem~4.2]{daners:15:ppe} it follows that $\lim_{\gamma\to\infty}\mu_1(\gamma\tilde b)<\infty$. Hence also $\mu_1^*(b)<\infty$.
\end{proof}

\begin{remark}
  \label{ex:mu-finite}
  The first example for $\mu_1(b)$ to be finite appears in \cite{du:12:ple}. A general criterion that is at least close to necessary is given in \cite{daners:15:ppe} and was used in Lemma~\ref{lem:mu-star-finite}: Assume that $b\in L^\infty(\Omega\times[0,T])$ and extend it $T$-periodically in $t$ to the infinite strip $\Omega\times\mathbb R$. Assume that $Q_0$ as defined in \eqref{eq:b-zero-set} and that $\supp(b)$ are both non-empty. According Lemma~\ref{lem:mu-star-finite} we have that $\mu^*(b)<\infty$ if there exists a \emph{forward moving} $T$-periodic path in $Q_0$. The optimality of this condition for $\mu^*(b)<\infty$ is discussed in \cite[Section~5]{daners:15:ppe} with a counter example if that condition is violated. A further counter example appears in \cite{lopez-gomez:20:pzp}. We note that for $\mu^*(b)$ to be finite it is not necessary that the zero set of $b$ has non-empty interior. An example in the stationary case (which is periodic with any period) is given in \cite[Remark~4.1]{daners:18:gdg}.
\end{remark}

The assumptions in Lemma~\ref{lem:mu-star-finite} guarantee that $\mu^{*}(b)<\infty$ and in turn, by Theorem~\ref{thm:limit-support}, gives the existence of a sequence of eigenfunctions that have a limit in $H$. The same argument as in the proof of \cite[Theorem~4.2]{daners:15:ppe}, shows that $\varphi_\infty$ is a weak solution of
\begin{displaymath}
  \frac{\partial\varphi_\infty}{\partial t}+\mathcal A(t)\varphi_\infty=\mu^*(b)\varphi_\infty\qquad\text{in }Q_0.
\end{displaymath}
It is important point out differences with the elliptic case such as treated in \cite{daners:18:gdg}. In that case every connected component of $\left\{x\in\Omega\colon b(x)=0\right\}$ is either contained in the support of the eigenfunction of the equivalent limit eigenvalue problem, or it has empty intersection with that support. Let us explore an example in the periodic-parabolic problem, where that is not the case.

First we define some sets that will be convenient here and in later sections. Take all assumptions as in Lemma~\ref{lem:mu-star-finite} and set
\begin{equation}
  \label{eq:Qb-sections}
  \Omega_t:=\{x\in \Omega\colon (x,t)\in Q_0\}
\end{equation}
for $t\in\mathbb{R}$. Assume now that $C$ is the connected component of $Q_0$ containing $(x,0)$ and let $\tilde\Omega$ be the set of $y\in\Omega_T$ such that there exists $\beta_1\in C([0,T],\Omega)$ with $\beta_1(0)=x$ and $\beta_1(T)=y$. Let $\tilde Q_0\subseteq Q_0$ such that there exists $x\in\tilde\Omega$ and a function $\beta_{2}\in C([0,t],\Omega)$ such that $\beta_{2}(0)=x$ and $(\beta_{2}(s),s)\in Q_0$ for all $s\in[0,t]$. Then $\tilde Q_0$ is connected and as a consequence of \cite[Theorem~3.12]{daners:15:ppe} and by the $T$-periodicity of the limit eigenfunction $\varphi_\infty$, it follows that either $\varphi_\infty(x,t)>0$ for all $(x,t)\in \tilde Q_0$ or $\varphi(x,t)=0$ for all $(x,t)\in\tilde Q_0$.

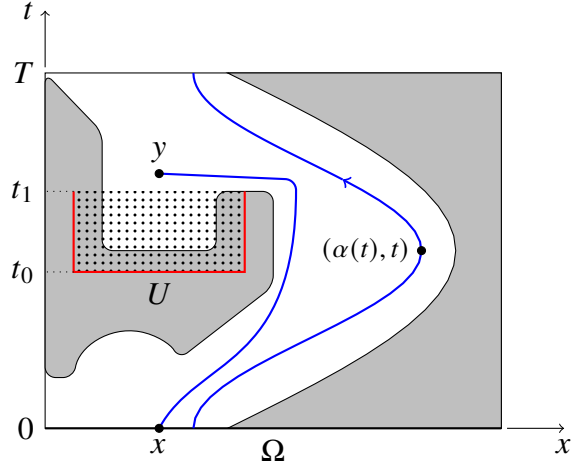
\begin{figure}[h!]
  \centering
  \begin{tikzpicture}[scale=1.5,
      declare function={
          T=pi;
          L=2;
          aa(\x)=L*(-0.35+pow(sin(deg(\x)),2));
        }]
    \draw (-L,0) rectangle (L,T);%
    \node at (-L,0) [left] {$0$};
    \node at (-L,T) [left] {$T$};
    \draw[->] (L+0.05,0) -- ++(0.5,0) node[below] {$x$};
    \draw[->] (-L,T+0.05) -- ++(0,0.5) node[left] {$t$};
    \draw[fill=lightgray] plot[domain=0:T] ({L*(-0.2+sin(deg(\x)))},\x) -- (L,T) -- (L,0) -- cycle;%
    \coordinate (O) at (-L,0);
    \coordinate (A1) at (-L/8,2*T/3);
    \coordinate (A2) at (-L/8,2.2*T/5);
    \coordinate (A3) at (-7*L/8,2.2*T/5);
    \coordinate (A4) at (-7*L/8,2*T/3);
    \coordinate (x) at (-L/2,0);
    \coordinate (y) at (-L/2,43*T/60);
    \draw[rounded corners,fill=lightgray] (-L,T/7) -- (-7*L/8,T/7) arc[start angle=170, end angle=30,radius=L/4] -- (0,2*T/5) -- (0,2*T/3) -- (-L/4,2*T/3) -- (-L/4,T/2) -- (-3*L/4,T/2) -- (-3*L/4,5*T/6) -- (-L,T) -- cycle;%
    \draw[thick] (-L,0) -- (L,0) node[midway,below] {$\Omega$};
    \draw[thick,blue,->] plot[domain=0:0.7*T] ({aa(\x)},\x);
    \draw[thick,blue] plot[domain=0.7*T:T] ({aa(\x)},\x);%
    \draw[fill] ({aa(T/2)},T/2) circle[radius=1pt] node[left,font=\footnotesize]  {$(\alpha(t),t)$};%
    \draw[thick,blue,rounded corners] (x) .. controls (-L/3,T/5) and (L/10,T/6) .. (L/10,21*T/30) -- (y);
    \draw[fill] (y) circle[radius=1pt] node[above] {$y$};
    \draw[fill] (x) circle[radius=1pt] node[below] {$x$};
    \draw[thick,red,pattern=dots] (A1) -- (A2) -- (A3) node[midway,below,black] {$U$} -- (A4);
    \draw[dotted] (A4) -- (A4-|O) node[left] {$t_1$};
    \draw[dotted] (A3) -- (A3-|O) node[left] {$t_0$};
  \end{tikzpicture}
  \caption{Example of $b$ with $\mu^*(b)<\infty$.}
  \label{fig:mu-finite}
\end{figure}

Figure~\ref{fig:mu-finite} shows an example of a weight function $b$. The shaded regions represent $Q_b$ and the white region including the dotted part represents $Q_0$. The set $\tilde Q_0$ is the white region excluding the dotted part. The figure also shows a curve given by $\alpha\in C([0,T],\Omega)$ inside the zero set with $\alpha(0)=\alpha(T)$. The figure also shows the path connecting points as required in the definition of $\tilde Q_0$. From our previous comments, $\varphi_\infty(x,t)>0$ for all $(x,t)\in\tilde Q_0$. Let us now focus our attention on the dotted region within $Q_0$. That region is part of a cylinder $U\times[t_0,t_1]$ for some open set $U\subseteq\Omega$ and $0\leq t_0<t_1\leq T$. Its parabolic boundary lies in $Q_b$ as shown in Figure~\ref{fig:mu-finite}. We know from Theorem~\ref{thm:limit-support} that $\varphi_\gamma(x,t)\to 0$ on $(U\times\{t_0\})\cup\partial U\times[t_0,t_1)$. Hence the parabolic maximum principle implies that $\varphi_\gamma(x,t)\to 0=\varphi_\infty(x,t)$ as $\gamma\to\infty$ for all $(x,t)\in U\times [t_0,t_1)$.

The periodic-parabolic spectrum gives us a way to prove the existence of periodic solutions to
\begin{equation}
  \label{eq:pp-ivp-abstract}
  \begin{aligned}
    \dot u+A(t)u & =f(t)  &  & \text{in (0,T]} \\
    u(0)         & =u(T).
  \end{aligned}
\end{equation}
whenever $f\in L^2((0,T),V')$. Since we are interested in positive solutions we establish the following theorem, see also \cite[Theorem~2.3]{daners:96:dpl}.

\begin{theorem}
  \label{thm:pp-inhomogeneous-existence}
  Assume that $1\in\varrho(U(T,0))$. Then \eqref{eq:pp-ivp-abstract} has a unique solution $u\in W(0,T;V,V')$. Moreover, if $f\geq 0$, then \eqref{eq:pp-ivp-abstract} has a positive solution if and only if $\spr(U(T,0))<1$ or equivalently $\mu_1(0)>0$.
\end{theorem}
\begin{proof}
  Let $w(t)$ be the unique solution of $\dot w+A(t)w=f(t)$ for $t\in(0,T)$ with $w(0)=0$. Such a solution exists due to Theorem~\ref{thm:linear-existence} and is positive if $f\geq 0$. As $1\in\varrho(U(T,0)$ then $(I-U(T,0))^{-1}$ is a positive operator. Define $u_0:=(I-U(T,0))^{-1}w(T)\geq 0$ and set $u(t):=U(t,0)u_{0}+w(t)$. Then $\dot{u}+A(t)u=f(t)$ for $t\in(0,T]$, $u(0)=u_{0}$ and
  \begin{align*}
    u(0) & = (I-U(T,0))^{-1}w(T)                       \\
         & = U(T,0)(I-U(T,0))^{-1}w(T) + w(T)          \\
         & = w(T) - w(T) + (I-U(T,0))^{-1}w(T) = u(0),
  \end{align*}
  so $u$ is a solution of \eqref{eq:pp-ivp-abstract}. Since $U(T,0)$ is irreducible and compact we finally note that $u_0=(I-U(T,0))^{-1}w(T)$ cannot be positive if $\spr(U(T,0))\geq 1$, see for instance \cite[Corollary~12.4]{daners:92:aee}.
\end{proof}

\section{Comparison of Principal Eigenfunctions}
\label{sec:eigenfunction-comparison}
The main result for this section is a comparison of the principal eigenfunctions of \eqref{eq:pp-evp-abstract}. The comparison we show here will be fruitful in showing the existence of sub and super-solutions of \eqref{eq:pp-logistic} on a non-smooth domain, which are crucial to derive the existence result Theorem~\ref{thm:main}. This is a counterpart of \cite[Theorem~6.1]{arendt:23:see} in a periodic-parabolic setting under somewhat different assumptions.

\begin{theorem}
  \label{thm:eigenfunction-comparison}
  Let $\varphi_{0}$ and $\varphi_{1}$ denote the principal eigenfunction of \eqref{eq:pp-evp-abstract} for $b_0$ and $b_1>0$, respectively. If $b_1$ has compact support and $\mu_1(b_0)\leq\mu_1(b_1)$, then there exists $c>0$ such that,
  \begin{equation*}
    \varphi_{0}<c\varphi_{1}
  \end{equation*}
  in $\Omega\times[0,T]$.
\end{theorem}

\begin{proof}
  Set $v_{c}:=\varphi_{0}-c\varphi_{1}$. We prove the existence of some $c_{0}>0$ such that $\varphi_{0}-c\varphi_{1}\leq 0$ on $\Omega\times[0,T]$ for all $c\geq c_{0}$. This is equivalent to showing $v_{c}^{+}=0$ for all $c\geq c_{0}$.

  Let us assume $v_{c}^{+}>0$ for all $c\geq 0$ and work towards a contradiction. We have the following equations,
  \begin{align}
    \label{eq:zero-ev}
    \dot{\varphi}_{0} + A(t)\varphi_{0} + b_0(t)\varphi_0
     & = \mu_{1}(0)\varphi_{0}                       \\
    \label{eq:m-ev}
    \dot{\varphi}_{1} + A(t)\varphi_{1}
     & = \mu_{1}(b_1)\varphi_{1} - b_1(t)\varphi_{1}
  \end{align}
  Subtract $c$ lots of \eqref{eq:m-ev} from \eqref{eq:zero-ev},
  \begin{equation*}
    \dot{v}_{c} + A(t)v_{c} + b_0(t)\varphi_0(t)
    = \mu_{1}(b_0)\varphi_{0}(t) - c\mu_{1}(m)\varphi_{1}(t) + cb_1(t)\varphi_{1}(t)
  \end{equation*}
  By assumption $\mu_{1}(b_0) < \mu_{1}(b_1)$ and $b_0(t)\varphi_0(t)\geq 0$ and thus
  \begin{equation*}
    \dot{v}_{c} + A(t)v_{c} \leq \mu_{1}(b_0)v_{c} + cb_1(t)\varphi_{1}(t).
  \end{equation*}
  By Proposition~\ref{prop:es-ev-monotonicity} we have that $\varphi_{k}\in BC(\Omega\times[0,])$ with $\varphi_k(x,t)>0$ for all $(x,t)\in\Omega\times[0,T]$. Since $\supp(b_1)\subseteq\Omega\times[0,T]$ is compact there exists $\delta>0$ such that $\varphi_{1}(x,t)>\delta$ for all $(x,t)\in\supp(b_1)$. Thus there exists $c_{0}>0$ with,
  \begin{equation*}
    b_1v_{c}^{+} = b_1(\varphi_{0}-c\varphi_{1})^{+} = 0
  \end{equation*}
  for all $c\geq c_{0}$.  We note that $v_{c}^{+}\in L^{2}((0,T),V)$ is a valid test function in the variational framework and so,
  \begin{equation}
    \label{eq:pp-evp-inequality}
    \begin{split}
      \int_{0}^{T}\left\langle\dot{v}_{c}(t),v_{c}^{+}(t)\right\rangle\,dt
       & + \int_{0}^{T}\aaa(t,v_{c}(t),v_{c}^{+}(t))\,dt                    \\
       & \leq\mu_{1}(b_0)\int_{0}^{T}\left(v_{c}(t),v_{c}^{+}(t)\right)\,dt
    \end{split}
  \end{equation}
  for all $c\geq c_{0}$. Applying \cite[Lemma~3.3]{daners:00:hke}, the $T$-periodicity of $v_c$ and the definitions of $\aaa(t,\cdot,\cdot)$ we have
  \begin{equation*}
    \begin{split}
      \int_{0}^{T}\aaa(t, & v_{c}^+(t),v_{c}^{+}(t))\,dt                                                    \\
                          & =\frac{1}{2}\left(\|v_{c}^{+}(T)\|^{2} - \|v_{c}^{+}(0)\|^{2}\right)
      + \int_{0}^{T}\aaa(t, v_{c}(t),v_{c}^{+}(t))\,dt                                                      \\
                          & \leq\mu_{1}(b_0)\int_{0}^{T}\left\langle v_{c}(t),v_{c}^{+}(t)\right\rangle\,dt
      =\mu_{1}(b_0)\int_{0}^{T}\|v_{c}^{+}(t)\|^2\,dt
    \end{split}
  \end{equation*}
  for all $c\geq c_0$. By Proposition~\ref{prop:form-properties} we can choose $\omega\in\mathbb{R}$ sufficiently large so that $\mu(0)+\omega>0$ and
  \begin{equation*}
    \alpha\int_{0}^{T}\|v_{c}^{+}(t)\|_{V}^{2}\,dt
    \leq (\mu_{1}(b_0)+\omega)\int_{0}^{T}\|v_{c}^{+}(t)\|_{H}^{2}\,dt
  \end{equation*}
  for all $c\geq c_0$. Hence there exists $K>0$ such that,
  \begin{equation}
    \label{eq:v-plus-estimate}
    \int_{0}^{T}\|v_{c}^{+}(t)\|_{V}^{2}\,dt
    \leq K\int_{0}^{T}\|v_{c}^{+}(t)\|_{H}^{2}\,dt
  \end{equation}
  By Rellich-Kondarchov embeddings \cite[Theorem~9.16]{brezis:11:fa}, $V\hookrightarrow L^{q}(\Omega)$ for some $q>2$, so there exists $K_{1}>0$ such that,
  \begin{equation}
    \label{eq:time-integral-bound}
    \int_{0}^{T}\left\| v_{c}^{+}(t) \right\|_{q}^{2}\,dt
    \leq K_{1}\int_{0}^{T}\left\| v_{c}^{+}(t) \right\|_{V}^{2}\,dt
  \end{equation}
  For any $\varepsilon\in(0,1)$ we can choose a compact set $A_{\varepsilon}\subseteq\Omega$ such that the measure $\left|A_{\varepsilon}^{c}\right|<\varepsilon$ and
  \begin{equation*}
    \supp(b_1)\subseteq A_\varepsilon\times [0,T].
  \end{equation*}
  Also, as $\varphi_1$ is continuous and $\varphi_{1}>0$ on $A_{\varepsilon}\times[0,T]$, there exists $c_{1}\geq c_{0}$ such that,
  \begin{equation*}
    v_{c}^{+} = (\varphi_{0}-c\varphi_{1})^{+}=0
    \quad\text{in }A_{\varepsilon}
  \end{equation*}
  on $A_\varepsilon\times[0,T]$ for all $c\geq c_{1}$. Set,
  \begin{equation*}
    w_{c}(t)
    := \dfrac{v_{c}^{+}(t)}{\left\| v_{c}^{+}(t) \right\|_{L^{2}((0,T),H)}}.
  \end{equation*}
  By H\"older's inequality, \eqref{eq:time-integral-bound} and \eqref{eq:v-plus-estimate} we see that
  \begin{align*}
    1 & =\int_0^T\|w_c(t)\|_H^2\,dt
    =\int_0^T\|w_c(t)1_{A_\varepsilon^c}\|_2^2\,dt                                                        \\
      & \leq|A_\varepsilon^c|^{2(\frac{1}{2}-\frac{1}{q})}\int_0^T\|w_c(t)\|_q^2\,dt
    \leq\varepsilon^{2(\frac{1}{2}-\frac{1}{q})}K_{1}\int_{0}^{T}\left\| w_{c}(t) \right\|_{V}^{2}\,dt    \\
      & \leq\varepsilon^{2(\frac{1}{2}-\frac{1}{q})}K_{1}K\int_{0}^{T}\left\| w_{c}(t) \right\|_H^{2}\,dt
    =\varepsilon^{2(\frac{1}{2}-\frac{1}{q})}K_{1}K                                                       \\
  \end{align*}
  for all $c\geq c_1$. If we choose $\varepsilon>0$ such that $\varepsilon^{2(\frac{1}{2}-\frac{1}{q})}K_{1}K<1$, then we have a contradiction. Hence there exists a $c>0$ such that $v_c^+=0$, that is $\varphi_0\leq c\varphi_1$ as claimed.
\end{proof}

\section{The periodic-parabolic logistic equation}
\label{sec:logistic}
In this section we prove the results in Theorem~\ref{thm:main}. We have introduced precise assumptions and an abstract framework for the linear problem in Section~\ref{sec:assumptions} and Section~\ref{sec:evolution-systems}. We continue to use that as well as the assumptions on the nonlinearity below.
\begin{assumption}
  \label{ass:g}
  Let $g\in C(\overline\Omega\times\mathbb R\times[0,\infty),\mathbb R)$ with
  \begin{itemize}
  \item $g(x,t,0)=0$ for all $(x,t)\in\overline\Omega\times\mathbb R$;
  \item $g(x,t+T,\xi)=g(x,t,\xi)$ for all $(x,t,\xi)\in\overline\Omega\times\mathbb R\times[0,\infty)$ ($T$-periodicity);
  \item $\dfrac{\partial g}{\partial\xi}\in C(\overline\Omega\times\mathbb R\times[0,\infty))$ and $\dfrac{\partial g}{\partial\xi}(x,t,\xi)>0$ for all $(x,t,\xi)\in\overline\Omega\times\mathbb R\times[0,\infty)$. \item $g(x,t,\xi)\to\infty$ as $\xi\to\infty$ uniformly with respect to $(x,t)$ in compact subsets of $\Omega\times[0,T]$.
  \end{itemize}
  Given a function $u\colon\overline\Omega\times[0,\infty)\to[0,\infty)$ the function $(x,t)\mapsto g(x,t,u(x,t))$ defines a function on $\overline\Omega\times[0,\infty)$. For convenience we often write $g(t,u)$ for the function given by $t\mapsto g(t,\cdot,u(\cdot,t))$ or even $g(u)$ for the function $g(\cdot\,,\cdot,u(\cdot\,,\cdot))$.
\end{assumption}
With these assumptions and notation the periodic parabolic evolution equation \eqref{eq:pp-logistic} can be written in the abstract form
\begin{equation}
  \label{eq:pp-logistic-abstract}
  \begin{aligned}
    \dot{u}+A(t)u & =\mu u - b(t)g(t,u)u &  & t\in(0,T] \\
    u(0)          & = u(T).
  \end{aligned}
\end{equation}
in $H$. We seek necessary and sufficient conditions for the existence of a non-trivial postive solution. We also consider solutions of the semi-linear evolution equation
\begin{equation}
  \label{eq:ivp-logistic-abstract}
  \begin{aligned}
    \dot{u}+A(t)u & =\mu u - b(t)g(t,u)u &  & t\in(0,T] \\
    u(0)          & = u_0.
  \end{aligned}
\end{equation}
with $u_0\in L^\infty(\Omega)$. By a solution to \eqref{eq:pp-logistic-abstract} or \eqref{eq:ivp-logistic-abstract} we mean an element $u\in W(0,T;V,V')$ with $g(\cdot\,,u)u\in L^2((0,T), V')$ that satisfies \eqref{eq:pp-logistic-abstract} or \eqref{eq:ivp-logistic-abstract}, respectively. We extend $A(t)$, $b(t)$ $T$-periodically to $\mathbb R$. Then any solution of \eqref{eq:pp-logistic-abstract} can be extended $T$-periodically to $t\in\mathbb R$ and any solution of \eqref{eq:ivp-logistic-abstract} can be extended to a solution for $t\geq 0$. We first prove that positive solution to \eqref{eq:pp-logistic-abstract} or \eqref{eq:ivp-logistic-abstract} are necessarily bounded.
\begin{lemma}
  \label{lem:solution-bounded}
  If $u$ is a non-trivial positive solution of \eqref{eq:pp-logistic-abstract} or \eqref{eq:ivp-logistic-abstract}, then $u\in L^\infty((0,T),L^\infty(\Omega))$ and $u(x,t)>0$ for all $(x,t)\in\Omega\times(0,T]$ and
  \begin{equation*}
    0\leq u(t)\leq e^{\mu t}U(t,0)u(0)
  \end{equation*}
  for all $t\geq 0$.
\end{lemma}
\begin{proof}
  Let $u$ be a positive solution of \eqref{eq:ivp-logistic-abstract} with $u(0)\in L^\infty(\Omega)$. Then, the function given by $w(t):=e^{\mu t}U(t,0)u(0)-u(t)$ satisfies the equation
  \begin{equation*}
    \begin{aligned}
      \dot{w}+A(t)w - \mu w & = b(t)g(t,u)u &  & t\in(0,T], \\
      w(0)                  & = 0.
    \end{aligned}
  \end{equation*}
  As $b(t)g(t,u)\geq 0$ it follows from Theorem~\ref{thm:linear-existence} that $w\geq 0$. Hence
  \begin{equation*}
    0\leq u(t)
    =e^{\mu t}U(t,0)u(0)-w(t)
    \leq e^{\mu t}U(t,0)u(0)
  \end{equation*}
  for all $t\geq 0$. In case of a solution of \eqref{eq:pp-logistic-abstract} we have that
  \begin{equation*}
    \|u(0)\|_\infty
    =\|u(T)\|_\infty
    \leq e^{\mu T}\|U(T,0)u(0)\|_\infty
    \leq ce^{\mu T}T^{-N/4}\|u(0)\|_2
    <\infty
  \end{equation*}
  by Proposition~\ref{prop:evolution-system}(U5). Hence in either case it follows from Proposition~\ref{prop:evolution-system}(U6) that there exists $C>0$ with
  \begin{equation*}
    \|u(t)\|_\infty
    \leq e^{\mu t}\|U(t,0)u(0)\|_\infty
    \leq Ce^{|\mu|T}\|u(0)\|_\infty
  \end{equation*}
  for all $t\in[0,T]$, showing that $u\in L^\infty((0,T),L^\infty(\Omega))$. Setting $m(t):=b(t)g(t,u(t))$ we have $m\in L^\infty([0,T],L^\infty(\Omega))$ and $u$ satisfies the linear equation
  \begin{equation*}
    \dot u+A(t)u+(m(t)-\mu)u=0
  \end{equation*}
  for $t\in[0,T]$ with $u(0)>0$. Hence, if $(U_m(t,s))_{(t,s)\in\Delta_T}$ is the evolution system associated with $(A(t)+m(t)-\mu)_{t\in[0,T]}$, then $u(t)=U_m(t,0)u_0$ for all $t\in[0,T]$ and thus $u(x,t)>0$ for all $(x,t)\in\Omega\times(0,T]$ by Propostion~\ref{prop:evolution-system}(U7).
\end{proof}
\begin{remark}
  \label{rem:logistic-eigenvalue}
  We note that from the above lemma the function $[t\mapsto b(t)g(t,u(t))]\in L^\infty((0,T),L^\infty)$ if $u$ is a solution of \eqref{eq:pp-logistic-abstract}. As a consequence, $u$ is a positive periodic-parabolic eigenvector for $(A(t)+b(t)g(t,u(t)))_{t\in[0,T]}$ and hence according to Theorem~\ref{thm:pp-evp-abstract-existence} we have
  \begin{equation*}
    \mu_1(bg(u))=\mu,
  \end{equation*}
  where we write $bg(u)$ as a short hand for the function $t\mapsto b(t)g(t,u(t))$. As before, given $m\in L^\infty((0,T),L^\infty(\Omega))$, we use the notation $\mu_1(m)$ for the principal periodic-parabolic eigenvalue of $\dot u+A(t)u+m(t)u=\mu u$, $u(0)=u(T)$.
\end{remark}
We next prove that any non-trivial positive solution to \eqref{eq:pp-logistic-abstract} is unique. The following tool is useful also for other purposes.
\begin{lemma}
  \label{lem:g-locally-lipschitz}
  Let $\xi_1,\xi_2\geq 0$. Then
  \begin{equation}
    \label{eq:g-locally-lipschitz}
    g(x,t,\xi_2)-g(x,t,\xi_1)=m(x,t,\xi_1,\xi_2)(\xi_2-\xi_1),
  \end{equation}
  where
  \begin{equation}
    \label{eq:ftc-g-positive}
    m(x,t,\xi_1,\xi_2):=\int_0^1\frac{\partial g}{\partial\xi}(x,t,\xi_1+s(\xi_2-\xi_1))\,ds
  \end{equation}
  and $m(x,t,\xi_1,\xi_2)>0$ for all $(x,t)\in \Omega\times[0,T]$.
\end{lemma}
\begin{proof}
  Given $\xi_1,\xi_1\geq 0$ the fundamental theorem of calculus implies that
  \begin{equation}
    \label{eq:ftc-g}
    \begin{split}
      g(x,t,\xi_2)-g(x,t,\xi_1)
       & =\int_0^1\frac{\partial g}{\partial\xi}(x,t,\xi_1+s(\xi_2-\xi_1))\,ds(\xi_2-\xi_1) \\ & =m(x,t,\xi_1,\xi_2)(\xi_2-\xi_1)
    \end{split}
  \end{equation}
  The assumptions on $g$ imply that $m(x,t,\xi_1,\xi_2)>0$ for all $(x,t)\in\Omega\times[0,T]$.
\end{proof}
\begin{proposition}[Uniqueness of solutions]
  \label{prop:pp-uniqueness}
  For every $\mu\in\mathbb R$, the periodic-parabolic problem \eqref{eq:pp-logistic-abstract} has at most one solution.
\end{proposition}
\begin{proof}
  Suppose that $u,v>0$ are two positive solutions of \eqref{eq:pp-logistic-abstract}. Taking into account Lemma~\ref{lem:g-locally-lipschitz} we see that $w:=v-u$, satisfies the equation
  \begin{equation*}
    \begin{aligned}
      \dot w+A(t)w+b(t)[g(t,v)+m(t,u,v)u]w & =\mu w &  & t\in[0,T] \\
      w(0)                                 & =w(T). &  &
    \end{aligned}
  \end{equation*}
  If $w\neq 0$ it follows that $\mu$ is a periodic-parabolic eigenvalue for the operators $(A(t)+b(t)[g(t,v)+m(t,u,v)v])_{t\in[0,T]}$. By \eqref{eq:ftc-g-positive} and Lemma~\ref{lem:solution-bounded} we have that
  \begin{equation*}
    b(t)[g(t,v)+m(t,u,v)u]>0
  \end{equation*}
  We also know from \eqref{eq:pp-logistic-abstract} that $\mu$ is a periodic-parabolic eigenvalue for the operators $(A(t)+b(t)+g(t,v(t)))_{t\in[0,T]}$ and thus by Theorem~\ref{thm:pp-evp-abstract-existence}
  \begin{equation*}
    \mu=\mu_1(b(t)g(t,v))<\mu_1(b[g(v)+m(u,v)u])\leq\mu.
  \end{equation*}
  This is impossible and thus $u=v$ as claimed.
\end{proof}
We next give a necessary condition for the existence of a positive solution.
\begin{proposition}[Necessary condition for existence]
  \label{prop:existence-necessary}
  Suppose that \eqref{eq:pp-logistic-abstract} has a non-trivial positive solution. Then $\mu_1(0)<\mu<\mu^*(b)$.
\end{proposition}
\begin{proof}
  We note that rearranging \eqref{eq:pp-logistic-abstract} we can write
  \begin{equation*}
    \begin{aligned}
      \dot{u}+A(t)u & + b(t)g(t,u)u =\mu u &  & t\in(0,T] \\
      u(0)          & = u(T),
    \end{aligned}
  \end{equation*}
  that is, $u$ is a principal eigenfunction corresponding to the eigenvalue $\mu_1(bg(\cdot\,,u))$. As $0<bg(\cdot\,,u)$ and $\gamma:=\|bg(\cdot\,,u)\|_\infty<\infty$, by Lemma~\ref{lem:solution-bounded}, we deduce from Theorem~\ref{thm:pp-evp-abstract-existence} that
  \begin{equation*}
    \mu_1(0)<\mu=\mu_1(bg(\cdot\,,u))\leq\mu_1(\gamma b)<\mu^*(b)
  \end{equation*}
  as claimed.
\end{proof}
We use the method of sub- and super-solutions to show that the necessary condtions for the existence of a positive solution of \eqref{eq:pp-logistic-abstract} in Proposition~\ref{prop:existence-necessary} is also sufficient. By a \emph{sub-solution} of \eqref{eq:pp-logistic-abstract} we mean a function $\underline{u}\in W(0,T;V,V')\cap L^\infty((0,T),L^\infty(\Omega))$ such that
\begin{equation}
  \label{eq:pp-subsolution}
  \begin{aligned}
    \dot{\underline{u}}+A(t)\underline{u} & \leq\mu \underline{u} - b(t)g(t,\underline{u})\underline{u} &  & t\in(0,T] \\
    \underline{u}(0)                      & \leq \underline{u}(T).
  \end{aligned}
\end{equation}
By a \emph{super-solution} of \eqref{eq:pp-logistic-abstract} we mean a function $\overline{u}\in W(0,T;V,V')\cap L^\infty((0,T),L^\infty(\Omega))$ such that
\begin{equation}
  \label{eq:pp-supersolution}
  \begin{aligned}
    \dot{\overline{u}}+A(t)\overline{u} & \geq\mu \overline{u} - b(t)g(t,\overline{u})\overline{u} &  & t\in(0,T] \\
    \overline{u}(0)                     & \geq \overline{u}(T).
  \end{aligned}
\end{equation}
We call $\underline{u}$ and $\overline{u}$ and ordered pair of sub- and super-solutions if $\underline{u}(0)\leq\overline{u}(0)$. These definitions are a more abstract version of \cite[Definition~21.1]{hess:91:ppb}. We show that a pair of ordered sub- and super-solutions is ordered on all $[0,T]$.

\begin{proposition}[Comparison]
  \label{prop:p-logistic-comparison}
  Suppose that $\underline{u}$ and $\overline{u}$ are an ordered pair of sub-and super-solutions of \eqref{eq:pp-logistic-abstract} and that $u,v$ are solutions of \eqref{eq:ivp-logistic-abstract} with $\underline u(0)\leq u(0)\leq v(0)\leq\overline u(0)$. Then $\underline{u}(t)\leq u(t)\leq v(t)\leq \overline{u}(t)$ for all $t\in[0,T]$.
\end{proposition}
\begin{proof}
  Setting $w:=\overline{u}-\underline{u}$ we can use \eqref{eq:g-locally-lipschitz} to write
  \begin{equation*}
    g(t,\overline{u})\overline{u}-g(t,\underline{u})\underline{u}
    =\left[m(t,\underline{u},\overline{u})\underline{u}+g(t,\overline{u})\right]w
  \end{equation*}
  Subtracting \eqref{eq:pp-subsolution} from \eqref{eq:pp-supersolution} we see that
  \begin{equation*}
    \begin{aligned}
      \dot w+A(t)w-\mu w+[g(t,\overline{u})+m(t,\underline u,\overline u)\underline{u}]w
                                            & =:f(t)\geq 0
                                            &              & t\in(0,T]   \\
      \overline{u}(0)-\underline{u}(0)=w(0) & \geq 0       &           &
    \end{aligned}
  \end{equation*}
  Since $g(t,\overline{u})+m(t,\underline u,\overline u)\underline{u}\in L^\infty((0,T),L^\infty(\Omega))$ and $f\in L^2((0,T),V')$ it follows from Theorem~\ref{thm:linear-existence} that $w\geq 0$, that is, $\underline{u}(t)\leq \overline{u}(t)$ for all $t\in[0,T]$. A similar argument holds if we replace the pair $\underline{u}$ and $\overline{u}$ by any combination of $\underline{u}$, $u$, $v$,and $\overline{u}$ with difference of initial conditions positive. This concludes the proof of the proposition.
\end{proof}
Knowing that solutions, if they exist, are trapped between sub- and super-solutions allows us to show the existence of a solution to the initial value problem \eqref{eq:ivp-logistic-abstract} for any initial condition between $\underline{u}(0)$ and $\overline{u}(0)$. We let
\begin{equation*}
  \tilde g(x,t,\xi):=\min\{g(x,t,\xi),g(x,t,\|\overline{u}\|_\infty)\}
\end{equation*}
for all $(x,t,\xi)\in\Omega\times\mathbb R\times[0,\infty)$. We also extend the function by zero for $\xi<0$. Taking into account \eqref{eq:g-locally-lipschitz} it follows that $\xi\to\tilde g(x,t,\xi)\xi$ is Lipschitz continuous on $\mathbb R$ uniformly with respect to $(x,t)\in\Omega\times[0,T]$. It then follows that we have a substitution operator $F:H\to H$ given by
\begin{equation*}
  F(t,u)(x,t):=\tilde g(x,t,u(x,t))u(x,t)
\end{equation*}
for any function $u\colon\Omega\times[0,T]\to\mathbb R$. Then it is easily checked that $F\in C(\mathbb R\times H,H)$ is Lipschitz continuous uniformly with respect to $t\in\mathbb R$. It follows that for every $u_0\in L^\infty(\Omega)$ the equation
\begin{equation*}
  \begin{aligned}
    \dot u+A(t)u-\mu u & =F(t,u) &  & t\geq 0 \\
    u(0)               & =u_0    &  &
  \end{aligned}
\end{equation*}
has a unique global (mild) solution, see for instance \cite[Section~16]{daners:92:aee} (the proofs work under our assumptions) or \cite[Theorem~4.5]{daners:05:pse}. Due to Proposition~\ref{prop:p-logistic-comparison} that solution coincides with the solution of \eqref{eq:ivp-logistic-abstract} for all $t\geq 0$ if $\underline{u}(0)\leq u_0\leq\overline{u}(0)$. A mild solution is a function $u\in C((0,T),H)$ such that
\begin{equation*}
  u(t)=U(t,0)u_0+\int_0^tU(t,\tau)(\mu u(\tau)-F(\tau,u(\tau))\,d\tau
\end{equation*}
for all $t\geq 0$. By Proposition~\ref{prop:evolution-system} this solution is also in $W(0,T;V,V')$.

We next construct a sub-solution for \eqref{eq:pp-logistic-abstract} in the same way as done in the proof of \cite[Theorem~28.1]{hess:91:ppb}.
\begin{lemma}[Existence of sub-solution]
  \label{lem:existence-subsolution}
  Suppose that $\mu>\mu_1(0)$ and let $\varphi_0$ be the principal periodic-parabolic eigenfunction as defined in Theorem~\ref{thm:pp-evp-abstract-existence}. Then there exists $\varepsilon_0>0$ such that $\varepsilon\varphi_0$ is a subsolution of \eqref{eq:pp-logistic-abstract} for every $\varepsilon\in(0,\varepsilon_0)$.
\end{lemma}
\begin{proof}
  For $\varepsilon>0$ we have that
  \begin{align*}
    \frac{d}{dt}(\varepsilon\varphi_0)
     & + A(t)(\varepsilon\varphi_0) = \mu_{1}(0)(\varepsilon\varphi_0)                                                                                                          \\
     & = \mu(\varepsilon\varphi_0) - b(t)g(t,\varepsilon\varphi_0)(\varepsilon\varphi_0) + \left[\mu_{1}(0) - \mu + b(t)g(t,\varepsilon\varphi_0)\right](\varepsilon\varphi_0).
  \end{align*}
  By assumption $\mu_1(0)-\mu<0$. Moreover, by assumption on $g$ and Proposition~\ref{prop:es-ev-monotonicity} we have that
  \begin{equation*}
    0\leq g(t,\varepsilon\varphi_0)
    \leq g(t,\varepsilon \|\varphi_0\|_\infty)
    \to 0
  \end{equation*}
  as $\varepsilon\to 0$ uniformly with respect to $(x,t)\in\Omega\times(0,T)$. Hence, there exists $\varepsilon_0>0$ such that
  \begin{equation*}
    \mu_{1}(0) - \mu + b(t)g(t,\varepsilon\varphi_0)<0
  \end{equation*}
  for all $\varepsilon\in(0,\varepsilon_0)$. Therefore,
  \begin{equation*}
    \frac{d}{dt}(\varepsilon\varphi_0) + A(t)(\varepsilon\varphi_0)
    <\mu(\varepsilon\varphi_0) - b(t)g(t,\varepsilon\varphi_0)(\varepsilon\varphi_0)
  \end{equation*}
  for all $\varepsilon\in(0,\varepsilon_0)$.
\end{proof}
We next prove the existence of a supersolution for $\mu<\mu^*(b)$. The construction follows the idea of \cite[Proposition~3.1]{daners:18:gdg} or \cite[Proposition~7.8]{arendt:23:see}. The construction is different from that in \cite[Section~5]{aleja:21:wpp} which makes use of more regularity properties.
\begin{proposition}[Existence of super-solution]
  \label{prop:existence-supersolution}
  Suppose that $\mu<\mu^*(b)$. For $\delta>0$ define
  \begin{equation*}
    \Omega_\delta:=\{x\in\Omega\colon\dist(x,\partial\Omega)\geq\delta\}
  \end{equation*}
  and let $b_\delta:=1_{\Omega_\delta}b$. Then the following statements hold.
  \begin{enumerate}[\normalfont (i)]
  \item There exists $\gamma>0$ such that $\mu<\mu_1(\gamma b_\delta)<\mu^*(b)$.
  \item If $\psi$ is the principal eigenfunction associated with $\mu_1(\gamma b_\delta)$ from \upshape{(i)}, then there exists $\kappa_0>0$ such that $\kappa\psi$ is a super-solution of \eqref{eq:pp-logistic-abstract} for all $\kappa\geq\kappa_0$.
  \end{enumerate}
\end{proposition}
\begin{proof}
  We start by proving (i). By definition of $\mu^*(b)$ we can choose $\gamma>0$ such that $\mu<\mu_1(\gamma b)$. It follows from Corollary~\ref{cor:spectral-radius-convergence}, Remark~\ref{rem:cutoff-convergence} and Lemma~\ref{lem:U-ev-relation} that $\mu_1(\gamma b)\leq \mu_1(\gamma b_\delta)$ and $\mu_1(\gamma b_\delta)\to\mu_1(\gamma b)$ as $\delta\to 0$. Hence there exists $\delta>0$ such that $\mu<\mu_1(\gamma b_\delta)<\mu^*(b)$.

  To prove (ii) let $\psi>0$ be a positive eigenfunction corresponding to $\mu_1(\gamma b_\delta)$. note that since $\mu<\mu_1(\gamma b_\delta)$ we have that
  \begin{align*}
    \frac{d}{dt}(\kappa\psi)
     & + A(t)(\kappa\psi) = \mu_1(\gamma b_\delta)(\kappa\psi)-\gamma b_\delta(t)(\kappa\psi) \\
     & >\mu(\kappa\psi)-b(t)g(t,\kappa\psi)(\kappa\psi)
    +\left[b(t)g(t,\kappa\psi)-\gamma b_\delta(t)\right](\kappa\psi).
  \end{align*}
  On $\Omega\setminus\Omega_\delta$ we have that
  \begin{equation*}
    b(t)g(t,\kappa\psi)-\gamma b_\delta(t)=b(t)g(t,\kappa\psi)\geq 0.
  \end{equation*}
  Since $\Omega_\delta\subseteq\Omega$ is compact we deduce from Proposition~\ref{prop:es-ev-monotonicity}(iv) that there exists $c>0$ such that $\psi(x,t)\geq c$ for all $x\in\Omega_\delta\times[0,T]$. Hence on $\Omega_\Omega$ we have $g(t,\kappa\psi)\geq g(t,\kappa c)\to\infty$ uniformly with respect to $(x,t)\in\Omega_\delta\times(0,T)$ Hence there exists $\kappa_0>0$ such that
  \begin{equation*}
    b(t)g(t,\kappa\psi)-\gamma b_\delta(t)>0
  \end{equation*}
  on $\Omega_\delta$ for all $\kappa\geq\kappa_0$. It follows that
  \begin{equation*}
    \frac{d}{dt}(\kappa\psi) + A(t)(\kappa\psi)
    >\mu(\kappa\psi)-b(t)g(t,\kappa\psi)(\kappa\psi)
  \end{equation*}
  for all $\kappa>\kappa_0$, showing that $\psi$ is a super-solution of \eqref{eq:pp-logistic-abstract}.
\end{proof}
There is no guarantee that the sub and super-solutions constructed above can be used to get an ordered pair. Usually the ordering is achieved by using $C^1$-regularity of the solutions and the Hopf boundary maximum principle. The key to overcome this restrictions on regularity is the eigenfunction comparison Theorem~\ref{thm:eigenfunction-comparison}. We use a monotone iteration scheme similar to \cite[Section~1 and~21]{hess:91:ppb}.
\begin{theorem}
  \label{thm:existence-of-periodic-solution}
  Let $\mu\in\left(\mu_{1}(0), \mu^{*}(b)\right)$ then there exists a unique non-trivial positive weak solution of \eqref{eq:pp-logistic}. Moreover, that solution is linearly stable.
\end{theorem}
\begin{proof}
  Let $\varepsilon\varphi$ and $\kappa\psi$ be the sub and super-solutions constructed in Lemma~\ref{lem:existence-subsolution} and Proposition~\ref{prop:existence-supersolution} for $\varepsilon\in(0,\varepsilon_0]$ and $\kappa\geq\kappa_0$. Since $b_\delta$ has compact support and $\mu_1(0)<\mu<\mu_1(\gamma b_\delta)$, Theorem~\ref{thm:eigenfunction-comparison} implies the existence of $\kappa_1\geq\kappa_0$ such that $\varepsilon\varphi\leq\kappa\psi$ for all $\varepsilon\in(0,\varepsilon_0]$ and all $\kappa\geq\kappa_1$.

  Fix a pair of sub and super-solutions $\underline u$ and $\overline u$. Let $w$ be the solution of \eqref{eq:ivp-logistic-abstract} with initial condition $\underline u(t)$. Then it follows by Proposition~\ref{prop:p-logistic-comparison} and induction that
  \begin{equation}
    \label{eq:iteration-bound}
    0<\underline{u}(t)
    \leq w(t+nT)
    \leq w(t+(n+1)T)
    \leq \overline{u}(t)
    \leq\|\overline u\|_\infty
  \end{equation}
  for all $n\in\mathbb N$. In particular
  \begin{equation*}
    u(t):=\lim_{n\to\infty}w(t+nT)
  \end{equation*}
  exists as a pointwise limit for all $t\in[0,T]$. The monotone convergence theorem implies convergence in $L^2((0,T),\Omega)$. Also, by the variation-of-constants formula and the $T$-periodicity
  \begin{align*}
    w( & (n+1)T)=U(T,0)w(nT)                                                              \\
       & \qquad+\int_0^TU(T,\tau)[\mu w(\tau+nT)-b(t)g(\tau,w(\tau+nT))]w(\tau+nT)\,d\tau
  \end{align*}
  for all $n\in\mathbb N$ and $t\in[0,T]$. By letting $n\to\infty$ and the dominated convergence theorem we have
  \begin{equation*}
    u(0)=U(T,0)+\int_0^TU(T,\tau)[\mu u(\tau)-b(t)g(\tau,u(\tau))]u(\tau)\,d\tau.
  \end{equation*}
  This shows that $u$ is a solution of \eqref{eq:pp-logistic-abstract}.

  For the linear stability we consider linearize the equation about $u_\mu$. That linearization can be written in the form
  \begin{displaymath}
    \begin{aligned}
      \frac{dw}{dt}+[A(t)-\mu+bg(x,t,u_\mu)]w+\frac{\partial g}{\partial\xi}g(x,t,u_\mu)u_\mu w&=0&&\text{for }t>0\\
      w(0)&=w_0.
    \end{aligned}
  \end{displaymath}
  We know that $\mu_1(-\mu+bg(u_\mu))=0$. Since $u_\mu(x,t)>0$ for all $(x,t)\in\Omega\times\mathbb R$, $g$ is strictly increasing and $b\neq 0$ it follows from \ref{prop:es-ev-monotonicity} that
  \begin{displaymath}
    \mu_1\left(-\mu+bg(u_\mu)+\frac{\partial g}{\partial\xi}g(u_\mu)u_\mu\right)>0.
  \end{displaymath}
  Hence, the linear stability follows from \cite[Theorem~22.2]{daners:92:aee}.
\end{proof}

\section{Smoothness with with respect to the parameter }
\label{sec:smoothness}
The non-trivial positive solution to \eqref{eq:pp-logistic-abstract} has some desirable properties with respect to the parameter $\mu$. We continue tho work with the same assumptions and framework as in Section~\ref{sec:logistic}. We have seen in Theorem~\ref{thm:existence-of-periodic-solution} that \eqref{eq:pp-logistic-abstract} has a unique solultion if and only if $\mu\in(\mu_{1}(0),\mu^{*}(b))$. One key property is the monotonicity.
\begin{theorem}
  \label{thm:smoothness-in-mu}
  Let $\mu\in(\mu_{1}(0),\mu^{*}(b))$ and let $u_\mu$ be the unique positive solution of \eqref{eq:pp-logistic-abstract}. Then $[\mu\mapsto u_\mu]\in C^1\bigl((\mu_1(0),\mu^*(b)),W(0,T;V,V')\bigr)$ and $u_{\mu}$ is increasing as a function of $\mu\in(\mu_{1}(0),\mu^{*}(b))$. Moreover, the derivative $v_{\mu}:=\frac{du_\mu}{d\mu}$ is the unique solution
  \begin{equation}
    \label{eq:linearized-logistic-equation}
    \begin{aligned}
      \dot{v}_{\mu} + A(t)v_{\mu} + b(t)\left[g(t,u_{\mu}) + \frac{\partial g}{\partial\xi}(t,u_{\mu})u_{\mu}\right]v_{\mu} - \mu v_{\mu} & = u_{\mu} &  & t\in(0,T], \\ v_{\mu}(0) & = v_{\mu}(T).
    \end{aligned}
  \end{equation}
\end{theorem}
\begin{proof}
  We first prove the monotinocity. Suppose $\mu_{1}(0)<\mu,\lambda<\mu^{*}(b)$ with $\mu\neq\lambda$. Setting
  \begin{equation*}
    v_\lambda:=\frac{u_{\lambda}-u_{\mu}}{\lambda-\mu}
  \end{equation*}
  we deduce from \eqref{eq:g-locally-lipschitz} that
  \begin{equation}
    \label{eq:pp-difference}
    \begin{aligned}
      \dot{v}_\lambda + A(t)v_{\lambda} +b(t)[g(t,u_\lambda)+m(t,u_\mu,u_\lambda)u_\mu)]v_\lambda-\lambda v_\lambda & = u_\mu       \\
      v_\lambda(0)                                                                                                  & =v_\lambda(T)
    \end{aligned}
  \end{equation}
  We note that $bm(t,u_\mu,u_\lambda)u_\mu>0$ and $\mu_1\left(b(t)g(t,u_\lambda)-\lambda\right)=0$ by Remark~\ref{rem:logistic-eigenvalue}. Hence by Theorem~\ref{thm:pp-evp-abstract-existence} we have that
  \begin{equation}
    \label{eq:linearized-ev-positive}
    \mu_1\bigl(bg(t,u_\lambda)-\lambda+bm(t,u_\mu,u_\lambda)u_\mu\bigr)>0
  \end{equation}
  for all $\lambda\in (\mu_1(0),\mu^*(b))$. Now Theorem~\ref{thm:pp-inhomogeneous-existence} implies that $v_\lambda$ is the unique solution to \eqref{eq:pp-difference} and since $u_\mu>0$ we have $v_\lambda>0$. In particular $u_\mu$ is increasing as a function of $\mu$.

  To prove the continuity fix $\mu\in(\mu_1(0),\mu^*(b))$. By the monotonicity it follows that $u_\lambda\to w$ in $L^2((0,T),H)$ and $u_\lambda(0)\to w(0)$ in $H$ as $\lambda\to \mu^+$. Fix $\delta>0$ such that $\mu+\delta<\mu^*(b)$ and consider $\lambda\leq\mu+\delta$. Then $0<u_{\lambda}\leq\|u_{\mu+\delta}\|_{\infty}$ and the family $\left(g(\cdot,u_\lambda)\right)_{\lambda\in(\mu_{1}(0),\mu+\delta)}$ is uniformly bounded in $L^\infty((0,T),L^\infty(\Omega))$. In particular, $g(\cdot,u_\lambda)\stackrel{*}{\rightharpoonup}g(\cdot,w)$ weak$^*$ in $L^{\infty}((0,T),L^\infty(\Omega))$. By Theorem~\ref{thm:linear-perturbation} $v$ is a solution to
  \begin{equation}
    \label{eq:pp-limit-equation}
    \begin{aligned}
      \dot w+A(t)w & =\mu w-b(t)g(t,w)w &  & t\in[0,T] \\
      w(0)         & =w(T)              &  &
    \end{aligned}
  \end{equation}
  and $u_\lambda\to w$ in $W(0,T;V,V')$. By the uniqueness of solutions from Proposition~\ref{prop:pp-uniqueness} it follows that $w=u_\mu$. A similar argument applies if $\lambda\to\mu^-$.

  For the differentiability define $\delta>0$ as above then
  \begin{equation*}
    m(t,u_{\mu},u_{\lambda}) \overset{*}{\rightharpoonup} \frac{\partial g}{\partial\xi}(t,u_{\mu})
  \end{equation*}
  weak$^*$ in $L^{\infty}((0,T),L^{\infty}(\Omega))$, similarly to $g(\cdot,u_{\lambda})$ as above. For any $\lambda\in(\mu_{1}(0),\mu+\delta)$ denote by $\left(U_{\lambda}(t,s)\right)_{(t,s)\in\Delta_{T}}$ the evolution system associated with
  \begin{equation*}
    \left(A(t)+b(t)\left[g(t,u_{\lambda})+m(t,u_{\mu},u_{\lambda})u_{\mu}\right]-\lambda\right)_{t\in[0,T]}.
  \end{equation*}
  Then by \eqref{eq:linearized-ev-positive} we have $1\in\varrho(U_{\lambda}(T,0))$. As in the proof of Theorem~\ref{thm:pp-inhomogeneous-existence} we then have $v_{\lambda}(t)=U_{\lambda}(t,0)w_{\lambda}(0)+y_{\lambda}(t)$ and in particular $v_{\lambda}(0) = (I-U_{\lambda}(T,0))^{-1}y_{\lambda}(T)$. By Theorem~\ref{thm:evolution-convergence} and a well-known perturbation result \cite[Theorem~IV~2.25]{kato:76:ptl}
  \begin{equation*}
    v_{\lambda}(0) = (I-U_{\lambda}(T,0))^{-1}y_{\lambda}(T) \to (I-U_{\mu}(T,0))^{-1}y_{\mu}(T) =: v_{0} \quad \text{in }H
  \end{equation*}
  as $\lambda\to\mu$, where $\left(U_{\mu}(t,s)\right)_{(t,s)\in\Delta_{T}}$ is the evolution system associated with the family of operators in \eqref{eq:pp-difference}. Hence, Theorem~\ref{thm:linear-perturbation} implies $v_{\lambda}\to v_{\mu}$ in $W(s_{0},T;V,V')$ for every $s_{0}\in(0,T]$. In particular $v_{\lambda}(0)=v_{\lambda}(T)\to v_{\mu}(T) = v_{\mu}(0)$ in $H$. Hence $v_{\lambda}\to v_{\mu}$ in $W(0,T;V,V')$ and $v_{\mu}$ is a solution of \eqref{eq:linearized-logistic-equation}. By \eqref{eq:linearized-ev-positive} and Theorem~\ref{thm:pp-inhomogeneous-existence} that solution is unique and positive. The continuity of the derivative follows from the continuity of $m_{\mu}$ as a function of $\mu$ and Theorem~\ref{thm:linear-perturbation}.
\end{proof}

\begin{corollary}
  \label{cor:bifurcation-of-periodic-solutions}
  Given $\mu^{*}(b)<\infty$ we have $\|u_{\mu}\|_{\infty}\downarrow 0$ as $\mu\downarrow\mu_{1}(0)$ and $\|u_{\mu}\|_{\infty}\uparrow\infty$ as $\mu\uparrow\mu^{*}(b)$.
\end{corollary}
\begin{proof}
  Suppose that $\mu\downarrow\mu_{1}(0)$. As in the continuity part of the proof of Theorem~\ref{thm:smoothness-in-mu} there exists $w$ such that $u_{\mu}\to w$ in $W(0,T;V,V')$ and $w$ satisfies the equation~\eqref{eq:pp-limit-equation} with $\mu=\mu_1(0)$. By Proposition~\ref{prop:existence-necessary} it follows that $w=0$ and in particular $u_\mu(0)\to 0$ in $H$ as $\mu\downarrow\mu_1(0)$. By Lemma~\ref{lem:solution-bounded} and Proposition~\ref{prop:evolution-system} we have that
  \begin{equation*}
    \begin{aligned}
      \|u_\mu(t)\|_\infty
      =\|u_\mu(t+T)\|_\infty
       & \leq e^{\mu t+T}\|U(t+T,0)u_\mu(0)\|_\infty           \\
       & \leq Ce^{2|\mu^{*}(b)|T}T^{-N/4}\|u_{\mu}(0)\|_2\to 0
    \end{aligned}
  \end{equation*}
  as $\mu\downarrow\mu_1(0)$ for all $t\in[0,T]$. Hence $u_\mu\to 0$ in $L^\infty((0,T),L^\infty(\Omega))$ as $\mu\downarrow\mu_1(0)$.
  In the case $\mu\uparrow\mu^{*}(b)<\infty$ we argue by contradiction. Suppose,
  \begin{equation*}
    \lim_{\mu\to\mu^{*}(b)}\|u_{\mu}\|_{\infty}=M<\infty
  \end{equation*}
  then the same argument as used in the continuity part of Theorem~\ref{thm:smoothness-in-mu} shows that $u_\mu\to w$ in $W(0,T;V,V')$ as $\mu\uparrow\mu^*(b)$ and $w$ solves \eqref{eq:pp-limit-equation} with $\mu=\mu^{*}(b)$. However this is a contradiction to Proposition~\ref{prop:existence-necessary} and so $\|u_{\mu}\|_{\infty}\to\infty$ as $\mu\uparrow\mu^*(b)$.
\end{proof}

The differentiability in Theorem~\ref{thm:smoothness-in-mu} coupled with the convergence and uniform boundedness of eigenfunctions of the equation \eqref{eq:pp-evp-abstract} allow us to specify where the blowup of solutions occurs. As an auxiliary problem we consider the problem
\begin{equation}
  \label{eq:periodic-torsion-function}
  \begin{aligned}
    \dot w+A(t)w & =1    &  & t\in[0,T] \\
    w(0)         & =w(T) &  &
  \end{aligned}
\end{equation}

\begin{proposition}
  \label{prop:uniform-blow-up}
  Let $w_{\gamma}$ denote a solution of
  \begin{equation}
    \label{eq:torsion}
    \begin{aligned}
      \dot{w} + A(t)w + \gamma b(t)w + \omega w & = 1          &  & t\in(0,T], \\
      w(0)                                      & = w(T)\in H,
    \end{aligned}
  \end{equation}
  where $\omega\in\mathbb{R}$. Suppose $\mu^{*}(b)<\infty$. If there exists $\mu_{0}\in(\mu_{1}(0),\mu^{*}(b))$ and $\gamma_{0}>0$ such that $u_{\mu_{0}}\geq w_{\gamma_{0}}$ then $u_{\mu}(x,t)\uparrow\infty$ as $\mu\to\mu^{*}(b)$ for all $(x,t)\in\Omega\times[0,T]$ where $\varphi_{\infty}(x,t)>0$.
\end{proposition}

\begin{proof}
  We take the family of uniformly bounded principal eigenfunctions, $(\varphi_{\gamma})_{\gamma\geq 0}$, from Proposition~\ref{prop:es-ev-monotonicity}. For each $\gamma\geq 0$ denote by $\left(U_{\gamma}(t,s)\right)_{(t,s)\in\Delta_{T}}$ the evolution system associated with the family $\left(A(t)+\gamma b(t) + \omega\right)_{t\in[0,T]}$. By the uniform bound on the family $(\varphi_{\gamma})_{\gamma\geq 0}$ and monotonoicity of $\mu_{1}(\gamma b)$ we have
  \begin{equation*}
    \begin{aligned}
      \varphi_{\gamma}(0) & = \left(I - U_{\gamma}(T,0)\right)^{-1}\int_{0}^{T}U_{\gamma}(T,s)\left(\mu_{1}(\gamma b) + \omega\right)\varphi_{\gamma}(s)\,ds                          \\
                          & \leq \left(I - U_{\gamma}(T,0)\right)^{-1}\int_{0}^{T}U_{\gamma}(T,s)\left(\mu^{*}(b) + \omega\right)C\, ds                                               \\
                          & = \left(\mu^{*}(b) + \omega\right)C\left(I - U_{\gamma}(T,0)\right)^{-1}\int_{0}^{T}U_{\gamma}(T,s)\,ds = \left(\mu^{*}(b) + \omega\right)Cw_{\gamma}(0),
    \end{aligned}
  \end{equation*}
  for all $\gamma\geq 0$. Hence, by Theorem~\ref{thm:linear-existence} we have $w_{\gamma}\geq M\varphi_{\gamma}$ for all $\gamma\geq 0$ where $M:=\left((\mu^{*}(b) + \omega)C\right)^{-1}$. From Proposition~\ref{prop:es-ev-monotonicity} $U_{\gamma}(t,s)$ is in $\gamma$ and so we have $w_{\gamma}$ is decreasing as $\gamma\to\infty$. Now, by assumption there exists $\mu_{0}\in(\mu_{1}(0),\mu^{*}(b))$ and $\gamma_{0}>0$ such that $u_{\mu_{0}}\geq w_{\gamma_{0}} \geq w_{\gamma} \geq M\varphi_{\gamma}$ for all $\gamma\geq\gamma_{0}$. Moreover, by Theorem~\ref{thm:smoothness-in-mu}, $u_{\mu}\geq M\varphi_{\gamma}$ for all $\mu\in[\mu_{0},\mu^{*}(b))$ and all $\gamma\geq\gamma_{0}$.

  Fix $\mu\in[\mu_{0},\mu^{*}(b))$ and choose $\gamma(\mu)>0$ such that
  \begin{equation*}
    \gamma(\mu) > \max\left\{\left\|g(\cdot,u_{\mu}) + \frac{\partial g}{\partial\xi}(\cdot,u_{\mu})u_{\mu}\right\|_{L^{\infty}(\Omega\times[0,T])}, \gamma_{0}\right\}
  \end{equation*}
  and $\mu_{1}(\gamma(\mu)b)>\mu_{0}$. By assumptions on $g$ and $\frac{\partial g}{\partial \xi}$, we have from Corollary~\ref{cor:bifurcation-of-periodic-solutions} that $\gamma(\mu)\to\infty$ as $\mu\to\mu^{*}(b)$. From Theorem~\ref{thm:smoothness-in-mu} we have that $v_{\gamma}:=\frac{d u_{\mu}}{d\mu}$ exists. Let $\left(U_{\mu}(t,s)\right)_{(t,s)\in\Delta_{T}}$ denote the evolution system associated with the problem $v_{\mu}$ satisfies. An application of Proposition~\ref{prop:es-ev-monotonicity} gives $U_{\mu}(T,0)>U_{\gamma}(T,0)$ for all $\gamma\geq\gamma(\mu)$, so
  \begin{equation*}
    \begin{aligned}
      v_{\mu}(0) & = \left(I - U_{\mu}(T,0)\right)^{-1}\int_{0}^{T}U_{\mu}(T,s)u_{\mu}(s)\,ds                                   \\
                 & \geq \left(I - U_{\mu}(T,0)\right)^{-1}\int_{0}^{T}U_{\mu}(T,s)M\varphi_{\gamma}(s)\,ds                      \\
                 & > \left(I - U_{\gamma}(T,0)\right)^{-1}\int_{0}^{T}U_{\gamma}(T,s)M\varphi_{\gamma}(s)\,ds = M\varphi_{g}(0)
    \end{aligned}
  \end{equation*}
  for all $\gamma\geq\gamma(\mu)$. Then,
  \begin{equation*}
    \begin{aligned}
      \dot{v}_{\mu} + A(t)v_{\mu} + \gamma b(t)v_{\mu} - \mu v_{\mu} & > u_{\mu}                                                                                                                                               \\
                                                                     & \geq \frac{M}{(\mu_{1}(\gamma(\mu)b) - \mu)}(\mu_{1}(\gamma(\mu)b) - \mu)\varphi_{\gamma}                                                               \\
                                                                     & = \frac{M}{(\mu_{1}(\gamma(\mu)b) - \mu)}\left(\dot{\varphi}_{\gamma} + A(t)\varphi_{\gamma} + \gamma b(t)\varphi_{\gamma} - \mu\varphi_{\gamma}\right)
    \end{aligned}
  \end{equation*}
  and hence
  \begin{equation*}
    v_{\mu} \geq \frac{M}{(\mu_{1}(\gamma(\mu)b) - \mu)}\varphi_{\gamma} \quad \forall\gamma\geq\gamma(\mu).
  \end{equation*}
  Choosing a sequence $(\varphi_{\gamma_{n}})$ so that Theorem~\ref{thm:limit-support} holds we have
  \begin{equation*}
    v_{\mu}(x,t) \geq \frac{M}{(\mu_{1}(\gamma_{n}(\mu)b) - \mu)}\varphi_{\gamma_{n}}(x,t) \to \infty \quad \text{as} \quad \mu\to\mu^{*}(b)
  \end{equation*}
  for any $(x,t)\in\Omega\times[0,T]$ where $\varphi_{\infty}(x,t)>0$ and so $u_{\mu}(x,t)\uparrow\infty$.
\end{proof}

\begin{remark}
  \label{rem:torsion}
  The existence of $w_{\gamma}$ in Proposition~\ref{prop:uniform-blow-up} is guaranteed by taking $\omega$ sufficiently large. If $\mu_{1}(\gamma b + \omega)>0$ then by Theorem~\ref{thm:pp-inhomogeneous-existence} there exists a positive solution of \eqref{eq:torsion}. In particular, we can choose $\omega\in\mathbb{R}$ independent of $\gamma$ due to the monotonicity of $\mu_{1}(\gamma b)$.
\end{remark}

\section{Local boundedness of blow up solutions}
We saw in Corollary~\ref{cor:bifurcation-of-periodic-solutions} that the solution $u_\mu$ of \eqref{eq:pp-logistic-abstract} blows up as $\mu\uparrow\mu^*(b)$ on the set where the limit eigenfunction $\varphi_\infty$ from Theorem~\ref{thm:limit-support} is positive. We now show that at least in some special case, the solutions $u_\mu$ of \eqref{eq:pp-evp-abstract} have a finite limit as $\mu\to\mu^*(b)$ if $\mu^*(b)<\infty$. We work with the following assumptions.
\begin{assumption}
  \label{assum:blow-up}
  We make the following assumptions for the remainder of this section:
  \begin{enumerate}[({B}1)]
  \item By a non-trivial positive solution of \eqref{eq:pp-logistic} we mean the case where
    \begin{equation*}
      \mathcal{A}(x,t):=-\alpha(t)\Delta,
    \end{equation*}
    where $\alpha\in L^\infty(\mathbb R)$ is $T$-periodic such that there exist constants $\alpha_0,\alpha_1>0$ with $\alpha_0\leq\alpha(t)\leq\alpha_1$ for almost every $t\in\mathbb R$. Under this assumption
    \begin{equation*}
      \aaa(t,u,v) = \alpha(t)(\nabla u,\nabla v)_{H} \quad \forall u,v\in V.
    \end{equation*}
  \item In addition to the usual conditions, we assume there exist $c>0$ and $p\in (2,\infty)$ such that
    \begin{equation*}
      g(x,t,\xi)\geq c\xi^{p-1}
    \end{equation*}
    for all $(x,t,\xi)\in \Omega\times\mathbb R\times[0,\infty)$.
  \end{enumerate}
\end{assumption}

As a result of the monotonicity of solutions $u_{\mu}$ (Theorem~\ref{thm:smoothness-in-mu}), with respect to $\mu$, we have that
\begin{equation}
  \label{eq:limit-solution}
  \lim_{\mu\to\mu^{*}(b)}u_{\mu}(x,t) = u_{\infty}(x,t)\in(0,\infty]
\end{equation}
exists, so we can consider the sets
\begin{equation*}
  \tilde{Q}_{\infty} := \operatorname{int}\left\{(x,t)\in \Omega\times\mathbb{R} \colon u_{\infty}(x,t)<\infty\right\}
\end{equation*}
and
\begin{equation*}
  Q_{\infty} := \tilde{Q}_{\infty}\cap\left(\Omega\times[0,T]\right).
\end{equation*}
We will show later that $Q_{b}\subseteq Q_{\infty}$, possibly with strict inclusion. For now we will derive a Sobolev estimate which will be used to show $u_{\infty}$ is a Sobolev function satisfying a PDE locally in $Q_{\infty}$.
\begin{lemma}
  \label{lem:local-sobolev-bound}
  Given two sub-cylinders $Q_{i}:=\Omega_{i}'\times(s_{i},t_{i})$, $i=1,2$, such that
  \begin{equation*}
    Q_{1}\Subset Q_{2}\Subset \Omega\times(0,T)
  \end{equation*}
  we have
  \begin{equation}
    \label{eq:H-L2-estimate}
    \|u_{\mu}\|_{L^{2}((s_{1},t_{1}),H^{1}(\Omega'_{1})} \leq\frac{C}{\alpha_0}\sqrt{\mu^{*}(b)+1}\|u_{\mu}\|_{L^{2}(Q_{2})}
  \end{equation}
  for every $\mu\in(\mu_{1}(0),\mu^{*}(b))$, where $C$ depends on $\dist(Q_{1},\mathcal{P}(Q_{2}))$ and $\alpha_1$, where $\mathcal{P}(Q_2)$ is the parabolic boundary of $Q_2$.
\end{lemma}

\begin{proof}
  We begin by showing an estimate on any sub-cylinder $\Omega'\times(s,t)\Subset\Omega\times(0,T)
  $ that depends on the choice of test function. For convenience set $Q':=\Omega'\times(s,t)$. Fix $v\in C_{c}^{\infty}(Q')$. For a given $u\in V$ we have the following identity,
  \begin{equation*}
    \aaa(vu,vu) = \aaa(u,v^{2}u) + (u^{2},|\nabla v|^{2})_{L^{2}}.
  \end{equation*}
  We also have from \eqref{eq:integration-by-parts}
  \begin{equation*}
    0 = \int_{s}^{t}\left\langle\frac{d}{d\tau}(vu_{\mu}),vu_{\mu}\right\rangle\,d\tau  = \int_{s}^{t}\left\langle\dot{v},vu_{\mu}^{2}\right\rangle\,d\tau + \int_{s}^{t}\left\langle\dot{u}_{\mu},v^{2}u_{\mu}\right\rangle\,d\tau
  \end{equation*}
  Now,
  \begin{equation*}
    \begin{aligned}
      \alpha_0\int_{s}^{t}\left\|\nabla(vu_{\mu})\right\|_{L^{2}(Q')}^{2} \,d\tau
       & \leq \int_{s}^{t}\aaa(\tau,vu_{\mu},vu_{\mu})\,d\tau                                                                                                 \\
       & = \int_{s}^{t}\aaa(\tau,u_{\mu},v^{2}u_{\mu})\,d\tau + \alpha_1\int_{s}^{t}(u_{\mu}^{2},|\nabla v|^{2})_{L^{2}}\,d\tau                               \\
       & = \mu\int_{s}^{t}(u_{\mu},vu_{\mu}^{2})_{L^{2}}\,d\tau - \int_{s}^{t}\left(b(\tau)g(\tau,u_{\mu})u_{\mu},v^{2}u_{\mu}\right)_{L^{2}}\,d\tau          \\
       & \hspace{15mm} - \int_{s}^{t}\left\langle\dot{u}_{\mu},v^{2}u_{\mu}\right\rangle\,d\tau + \alpha_1\|u_{\mu}\nabla v\|_{L^{2}(Q')}^{2}                 \\
       & \leq \mu\|vu_{\mu}\|_{L^{2}(Q')}^{2} + \int_{s}^{t}\left\langle\dot{v},vu_{\mu}^{2}\right\rangle\,d\tau +\alpha_1\|u_{\mu}\nabla v\|_{L^{2}(Q')}^{2} \\
       & \leq \left(\mu^{*}(b)\|v\|_\infty + \|\dot{v}v\|_{L^{\infty}(Q')} +\alpha_1\|\nabla v\|_{L^{\infty}(Q')}^{2}\right)\|u_{\mu}\|_{L^{2}(Q')}^{2}.
    \end{aligned}
  \end{equation*}
  Choosing $v\in C_{c}^{\infty}(Q_{2})$ such that $0\leq v\leq 1$ and $v\equiv 1$ on $Q_{1}$ the inequality \eqref{eq:H-L2-estimate} follows.
\end{proof}

\begin{theorem}
  \label{thm:ple-limit-equation}
  Let $u_{\infty}$ be as in \eqref{eq:limit-solution}. Then, $u_\infty$ is a local weak solution of the equation
  \begin{equation}
    \label{eq:ple-limit-equation}
    \partial_{t}u_{\infty} - \alpha(t)\Delta u_{\infty} = \mu^{*}(b)u_{\infty} - b(x,t)g(x,t,u_{\infty})u_{\infty}
  \end{equation}
  in $Q_\infty$.
\end{theorem}

\begin{proof}
  Fix an arbitrary $(x,t)\in Q_\infty$ and choose a sub-cylinder $Q_1=\Omega_1\times(s_1,t_1)$ such that its closure is contained in $Q_\infty$. Let $v\in C_{c}^{\infty}(Q_1)$. By Lemma~\ref{lem:local-sobolev-bound} it follows that $u_{\mu}$ is bounded in $L^2((s_1,t_1),H^1(\Omega_1))$, so there exists a sequence $(u_{\mu_{n}})$ such that $u_{\mu_{n}}\rightharpoonup u_{\infty}$ in $L^{2}((s_{1},t_{1}),H^{1}(\Omega_{1}'))$ weakly as $n\to\infty$. As $u_{\infty}$ is a unique limit the convergence holds for the family $(u_{\mu})_{\mu\geq 0}$ as $\mu\to\mu^{*}(b)$.  Let $\theta\in C_c^\infty(Q_1)$. Then,
  \begin{equation*}
    \begin{aligned}
      -\int_{s_{1}}^{t_{1}}(u_{\mu},\partial_{t}\theta)_{L^{2}}\,d\tau
       & + \int_{s_{1}}^{t_{1}}\aaa(\tau,u_{\mu},\theta)\,d\tau                                     \\
       & = \int_{s_{1}}^{t_{1}}(\mu u_{\mu} - b(\tau)g(\tau,u_{\mu})u_{\mu},\theta)_{L^{2}}\,d\tau.
    \end{aligned}
  \end{equation*}
  By the weak convergence of $(u_{\mu})_{\mu\geq 0}$ we have in the limit $\mu\to\mu^{*}(b)$
  \begin{equation*}
    \begin{aligned}
      -\int_{s_{1}}^{t_1}(u_{\infty},\partial_{t}\theta_{j})_{L^{2}}\,d\tau
       & + \int_{s_{1}}^{t_{1}}\aaa(\tau,u_{\infty},\theta)\,d\tau                                                 \\
       & = \int_{s_{1}}^{t_{1}}(\mu^{*}(b)u_{\infty} - b(\tau)g(\tau,u_{\infty})u_{\infty},\theta)_{L^{2}}\,d\tau.
    \end{aligned}
  \end{equation*}
  Hence $u_\infty$ is a local weak solution of \eqref{eq:ple-limit-equation} as claimed.
\end{proof}

The set $Q_{\infty}$ by definition is the set on which the periodic solutions $u_{\mu}$ have a local $L^\infty$-bound as $\mu\uparrow\mu^*(b)$. The aim is to show that $Q_b\subseteq Q_\infty$. For that purpose we construct local super-solution for \eqref{eq:pp-logistic} that are independent of $\mu\leq\mu^*(b)$. These local super-solutions on sub-cylinders of $Q_b$ constructed in a very similar way as those in \cite{du:12:pleb}. There are similar results for the stationary problem, see for instance \cite{cirstea:02:eub,du:06:ost,lopez:16:mpe}.
\begin{proposition}
  \label{prop:local-bound}
  Suppose that Assumption~\ref{assum:blow-up} is satisfied. Then the family of periodic solutions $(u_{\mu})_{\mu\in(\mu_1(0),\mu^*(b))}$ is bounded in $L^{\infty}_{loc}(Q_{b})$.
\end{proposition}
\begin{proof}
  The proof relies on the comparison of solutions $u_{\mu}$ with a super-solution to a boundary blow-up problem on strongly included sub-cylinders of $Q_{b}$. Let $U\Subset\Omega$, with $C^{2}$ boundary, and $(s,t)\subseteq[0,T]$ such that $\overline{U}\times[s,t]\subseteq Q_{b}$.  Let $c>0$ and $p>2$ be as in the assumptions.  By definition of $Q_b$ there exists $B>0$ such that $b(x,\tau)\geq B$ for all $(x,t)\in\overline U\times [s,t]$. For every $\mu\in(\mu_{1}(0),\mu^{*}(b))$ we have that $u_{\mu}$ is a sub-solution of
  \begin{equation}
    \label{eq:blow-up}
    \begin{aligned}
      \partial_{t}u - \alpha(t)\Delta u & = \mu^{*}(b)u - Bcu &  & \text{in }U\times(s,T^{*}], \\ u(x,t) & = u_{\mu}(x,t) &  & \text{on }\partial U\times(s,T^{*}],\\ u(x,s) & = u_{\mu}(x,s) &  & \text{in }U.
    \end{aligned}
  \end{equation}
  In order to construct a super-solution of \eqref{eq:blow-up} use the solutions of two problems. It is easily verified that the Bernoulli type differential equation
  \begin{equation*}
    \begin{aligned}
      \dot{z}               & = \mu^{*}(b) z - cBz^{p} &  & t>s, \\
      \lim_{t\to s^{+}}z(t) & = \infty,
    \end{aligned}
  \end{equation*}
  has a unique strictly positive solution. Moreover, it is know that the elliptic boundary blow-up problem
  \begin{equation*}
    \begin{aligned}
      -\Delta w & = \frac{\mu^{*}(b)}{\alpha_0} w - \frac{cB}{\alpha_1} w^{p} &  & \text{in }U,          \\
      w         & = \infty                                                    &  & \text{on }\partial U,
    \end{aligned}
  \end{equation*}
  also has a unique strictly positive continuous solution, see for instance \cite[Theorem~6.14 \& 6.18]{du:06:ost}. Setting $v(x,t):=w(x)+z(t)$ we then have
  \begin{align*}
    \partial_{t}v - \alpha(t)\Delta v & = \dot{z} - \alpha(t)\Delta w                                                                                 \\
                                      & = \mu^{*}(b)\left(\frac{\alpha(t)}{\alpha_0}w+z\right) - cB\left(\frac{\alpha(t)}{\alpha_1}w^{p}+z^{p}\right) \\
                                      & \geq \mu^{*}(b)(w+z) - cB(w^{p}+z^{p})                                                                        \\
                                      & \geq \mu^{*}(b)(w+z) - cB(w^{p}+z^{p})
    \geq \mu^{*}(b)v - cBv^{p}
  \end{align*}
  in $U\times(s,T^{*}]$, where the last inequality follows from Minkowski's inequality. Hence, $v$ is a super-solution of \eqref{eq:blow-up}. It follows from the weak maximum principle, see for instance \cite[Theorem~1]{aronson:67:lbs} that
  \begin{equation*}
    v(x,t) > u_{\mu}(x,t) \quad \text{for }(x,t)\in U\times(s,t]
  \end{equation*}
  for all $\mu>\mu^*(b)$, showing that the family of periodic solutions is locally bounded.
\end{proof}
The key to Proposition~\ref{prop:local-bound} is the existence of a blow-up solution on a strongly included sub-cylinder whose parabolic boundary lies in $Q_{b}$. One can then ask what occurs when $Q_{0}$ has regions which can be taken as part of the interior of a strongly included sub-cylinder whose parabolic boundary lies $Q_{b}$. In Figure \ref{fig:mu-finite} we have such an example given by the dotted region enclosed by the sub-cylinder $U\times(t_{0},t_{1}]$. After the transformation $v_{\mu}:=e^{-\mu t}u_{\mu}$ we have by Proposition~\ref{prop:local-bound} that there exists $M>0$ such that $v_{\mu} \leq Me^{|\mu^{*}(b)|T}$ in a neighbourhood of $\mathcal{P}(U\times(t_{0},t_{1}])$ in $Q_{b}$ and so, by the weak parabolic maximum principle,
\begin{equation*}
  \sup_{\overline{U}\times[t_{0},t_{1}]}v_{\mu} \leq Me^{|\mu^{*}(b)|T} \implies \sup_{\overline{U}\times[t_{0},t_{1}]}u_{\mu} \leq Me^{2|\mu^{*}(b)|T}.
\end{equation*}

In a more general scenario, where the set $Q_{0}$ and $Q_{b}$ may have a geometry that does not work so easily with a single sub-cylinder as in Figure \ref{fig:mu-finite}. We can take a finite collection of overlapping sub-cylinders with parabolic boundary in $Q_b$ or a cylinder below, for example see Figure \ref{fig:alcove-cover} for an illustration. It then follows that $Q_{b}\subseteq Q_{\infty}$ but depending on the geometry of $Q_{0}$ it is possible that $Q_{\infty}\cap Q_{0}\neq\emptyset$.

\begin{figure}[ht]
  \centering
  \begin{tikzpicture}[scale=1]
    \draw (-2.5,-2) node[left] {$0$} -- (1.5,-2) node[midway,below] {$\Omega$} -- (1.5,1.5) -- (-2.5,1.5) node[left] {$T$} -- cycle;%
    \draw[fill=lightgray] (-2,1) -- (-1.5,1) -- (-0.5,-0.5) -- (0.5,1) -- (1,1) -- (-0.5,-1.6) -- cycle;%
    \draw[thick,red,pattern=dots] (-1.6,1) |- (0.6,0.4) -- (0.6,1);%
    \draw[thick,red,pattern=dots] (-1.2,0.5) |- (0.2,-0.2) -- (0.2,0.5);%
    \draw[thick,red,pattern=dots] (-0.9,-0.1) |- (-0.1,-0.7) -- (-0.1,-0.1);%
  \end{tikzpicture}
  \caption{$Q_b$ (shaded) with overlapping sub-cylinders covering part of $Q_0$}
  \label{fig:alcove-cover}
\end{figure}
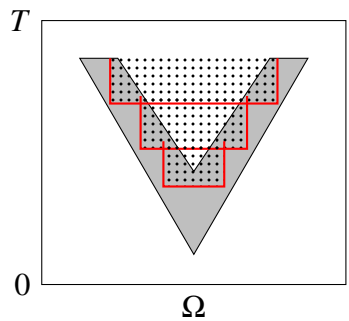

This provides a clear distinction between the behaviour of positive solutions of \eqref{eq:pp-logistic} and those of the corresponding elliptic logistic equation. As illustrated by the examples in Figure~\ref{fig:mu-finite} and Figure~\ref{fig:alcove-cover}, solutions may blow up only on part of a connected component of $Q_0$ in the periodic-parabolic case, whereas if there is blowup in for the solutions of the corresponding solution of the stationary equation, it is on the whole connected component of that zero set.

\end{document}